\documentclass{amsart}

\usepackage{fullpage} 

\usepackage{tikz} 
\usepackage{float} 
\usetikzlibrary{shapes, arrows.meta, positioning, patterns, calc}
\usetikzlibrary{automata, positioning, arrows}
\usepackage{amsmath}
\usepackage{mathrsfs}\overfullrule=5pt
\usepackage{amsfonts,epsfig,amsmath,amssymb,amscd,graphicx, amsmath,amsthm}
\usepackage[colorlinks=true]{hyperref}
\hypersetup{urlcolor=red, citecolor=green, linkcolor=red}
\usepackage{pstricks}
\usepackage[titletoc]{appendix}
\usepackage{tikz-cd}
\usepackage{tikz}
\usetikzlibrary{cd}
\title{Finite type as fundamental objects even non-single-valued and non-continuous}
\author{Zhengyu Yin}

\date{}

\input xy
\xyoption{all}
\allowdisplaybreaks[4]
\newcommand{\mesh}{\operatorname{mesh}}

\newtheorem{thm}{Theorem}[section]
\newtheorem{lem}[thm]{Lemma}

\newtheorem{prop}[thm]{Proposition}
\newtheorem{ques}[thm]{Question}
\newtheorem{cor}[thm]{Corollary}

\theoremstyle{definition}
\newtheorem{de}[thm]{Definition}
\newtheorem*{theorem*}{Theorem}
\newtheorem*{cor*}{Corollary}
\newtheorem{exam}[thm]{Example}
\newcommand{\FOC}{\mathcal{FOC}}
\newtheorem{rem}[thm]{Remark}
\newtheorem{cla}[thm]{Claim}

\numberwithin{equation}{section}

\address{School of Mathematics and Statistics, Hainan University.}
\email{yzy199707@gmail.com}
\begin{document}

\maketitle
\begin{abstract}
  In this paper, inspired by the elegant work of Good and Meddaugh \cite{GM} and the graph models for zero-dimensional systems developed by several authors, like Gambaudo and Martens \cite{GM06}, Shimomura \cite{Sh14}. We try to discover a connection among some objects, such as finite directed graph, shift of finite type and shadowing property by employing the Closed Graph Theorem for multivalued maps. 
From the perspective of structure theorems, we demonstrate that every closed relation (multivalued map) on a compact, totally disconnected space is represented as an inverse limit of finite directed graph homomorphisms satisfying the Mittag-Leffler condition. Moreover, from dichotomy-theorem point of view, we prove that an inverse limit of finite directed graph homomorphisms possesses the shadowing property if and only if its induced space of infinite graph walks (as a shift of finite type) satisfies the Mittag-Leffler condition. As an application, a question raised by Boro\'nski,  Bruin and Kucharski \cite{BBK} is also concerned.
Furthermore, we show that under a multivalued dynamical system, the resulting dynamical behaviors exhibit greater diversity and counterintuitively compared to those observed in single-valued continuous systems.
\end{abstract}

\section{Introduction}
Let $X$ and $Y$ be two topological spaces. A map $f$ from $X$ to $Y$ is a relation with the structural information: each element in $X$ is an initial vertex and each member in $Y$ is a terminal vertex linked by some edge in the graph $G_f$ of $f$. Naturally, $f$ carries the information of geometry (structure) and graphs. Particularly, if $f$ is a map from a compact Hausdorff space $X$ to itself, then $f$ also carries the dynamical properties like chaos, entropy, recurrence, stability etc. Usually, we focus on the autonomous action on a system, that is, one input corresponds to one output, while from the graph point of view, the single-valued continuous self map on a phase space is just a special case in the multivalued classes. For example, in the book \cite{Ak}, the concepts of recurrence with respect to a closed relation on a space are concerned, and the semiconjugacy between two closed relations (general closed graphs) resemble the graph homomorphisms betweeen two finite graphs is defined. Thus, it is natural to discuus the universal construction of inverse system consisting of such closed graphs. In this paper, we shall show that if the inverse system of closed graphs satisfies Mittag-Leffler condition, then the inverse limit also forms a closed graph on some space.

Multi-valued mappings possess richer phenomena than single-valued ones.  For example, for a continuous single-valued map $f:X\to X$, then $X$ is homeomorphic to its inverse limit which is defined by $\{(x_i)_{i\in \mathbb{N}}:f(x_i)=x_{i+1}\}$. Thus, they share the same topological properties, like connectness, topological dimension and so on. But for multivalued ones, there are examples showing that the connectivity between points decreased, thus, connected phase space may generate totally disconnected inverse limit, see e.g., \cite{IM}, \cite{In}. For dynamical systems, it is known that there is no \textit{expansive} homeomorphism on the closed unit $[0,1]$ \cite{Br62}, but $[0,1]$ adimts expansive upper-semisontinuous map \cite{Wi}.  Interestingly, an example in this paper shows that the action of an expansive multivalued map on a continuum can be viewed as a flipping coin motion.

For multivalued maps, the problem of finding a continuous selection for set-valued mappings has been a central theme in set-valued analysis over the past few decade, see e.g., \cite{Mic56} and a survey \cite{RS} and these results are applied to Game Theory \cite{YP}, Theory of Optimal Control \cite{Fi}, Approximation Theory \cite{De} and so on. Recently,
a growing body of literature has begun to consider the action of a system such that each starting point carries multiple choice. Thus, it is natural to consider the chaos behavior for such systems, such as Fixed points, e.g., \cite{Dai17}, topological entropy, e.g., \cite{AK, BEK, CMM, KT, RT, WZZ}, recurrence properties, e.g., \cite{AFL, HL}, invariant measures, e.g., \cite{Ak, MA}, and systemetic stability, like \cite{MMT}. Discussion with multivalued maps also helps us to understand the strucutre of a dynamical system in the classical sense, see e.g., \cite{DFLX, DX, Gl}. 

The motivation for us to consider the shadowing porperty for multivalued maps is original from \cite{MMT}, the authors in \cite{MMT} discussed the stability for set-valued maps and leave a question: whether the single-valued stability (in single-valued category) can be extended to multivalued stability (in multivalued or graph category), we try to get it but failed. But we notice that stability implies shadowing property. Due to the \textbf{closed graph theorem} for upper semicontinuous multivalued maps, there are two objects for a map $f:X\to 2^X$: First, the underlying phase space $X$ together with its time $1$ motion, that is, if we start with a point $x$ at some time $t$, then we know the possible motion at time $t+1$. Second, if we  consider the infinite time for $f$, then we get its orbit spaces with the constrained $f:X\to 2^X$. So from these two points of view, a property $\mathbb{P}$ given on $f:X\to 2^X$ actually induces property $\mathbb{P}_o$ on its orbit space or has counterexample like in \cite{IM}. Due to the limitations of our knowledge, we do not know that whether $f:X\to 2^X$ has shadowing property (cf. \cite{MMT, PR}) implying its infinite orbit space has (single-valued) shadowing property, that is, multivalued shadowing property generates single-valued shadowing property, so we discuss some situation in \cite{Yin}.

The shadowing property in symbolic systems naturally connects with the concept of shift of finite type (see e.g., \cite{Br, LM, PW1}). After reading these literature, we catch an interesting sign which links to the  multivalued maps: let $V$ be a finite set, and $E\subset V\times V$ a relation on $V$. Then $G=(V,E)$ generates a finite (direced) graph, in particular, if we endow $V$ with discrete topology, then the infinite walk produced by $G$ becomes a symbolic system, indeed, a shift of finite type. Now, if we define a map $f:V\to 2^V$ by $f(u)=\{v:(u,v)\in E\}$, we call it \textit{multivalued map of finite type}, then $f$ is a multivalued map on $V$ and its obrit space is a shift of finite type, thus, having shadowing property. In \cite{GM}, Good and Meddaugh prove that a single-valued zero-dimensional system has shadowing property if and only if it is conjugate to an inverse limit of shift of finite type which satisfies Mittage-Leffler condition. Thank to their elegent work, which leads us to consider the shadowing property for multivalued version. By reading their paper \cite{GM}, Good and Meddaugh actually state a \textbf{dichotomy-theorem}, that is, a single-valued zero-dimensional dynamical system is homeomorphic to an inverse limit of shift of finite type, but only the systems with shadowing property satisfying Mittag-Leffler condition. However, we find this statement fail for multivalued maps. Precisely, it is mentioned above that a multivalued map on a finite phase space automatically possess shadowing property, but a beautiful result \cite[Theorem 3.9]{Sh14} proved by T. Shimomura shows that every single-valued zero-dimensional dynamical system is conjugate to an inverse limit of multivalued maps of finite type, where the bonding maps between multivalued maps of finite type are surjective, thus, satisfying Mittag-Leffler condition. However, not every zero-dimensional dynamical system has shadowing property, so this implies that a multivalued system with shadowing property possessing deeper structure.

To deal with the multivalued maps, a problem is that how we can record a single-orbit efficiently, in \cite{GM06}, Gambaudo and Martens characterize single-valued minimal zero-dimensional system $X$ with some specific finite directed graphs, using graph theory, derive a \textbf{structure theorem} by approximating each point $x\in X$ with inverse limit of vertex sets it belongs to, which explicitly describes the geometric structure of the system. The graph representation method of a zero-dimensional system is also applied to describe a \textit{generic} continuous (homeomporphic) map over Cantor systems,  (cf. \cite{AGW}, \cite{BD}). Later,  T. Shimomura \cite{Sh14, Sh16, Sh20} further develope the graph representation method for zero-dimensional single-valued systems, and name the inverse limit with specific conditions \textbf{graph cover}. From the point of structure, the mutivalued maps and the graph with veteices and edges are in the same category. Thus, the closed graph theorem leads us to apply graph theory to study the structure or dynamical properties  for multivalued maps. In historical sequence, we consider the expansive multivalued maps with shadowing property (cf. \cite{PW1}) and obtain the first result, the relation between multivalued shadowing and shift of finite type. Meanwhile, the concept of \textit{generator} for expansive systems in \cite{KR} is generalized to multivalued version
\begin{theorem*}
Let $(X, f)$ be an expansive compact totally disconnected multivalued system. If $f$ has the shadowing property, then its orbit system is conjugate to a subshift of finite type.    
\end{theorem*}
Further, to establish the structure of a multivalued system with shadowing property on a zero-dimensional space, we apply the method of T. Shimomura \cite{Sh14,Sh16, Sh16} to a general closed graph on a general zero-dimensional space, that is, 
\begin{theorem*}
    $(X,f)$ is a compact totally disconnected Hausdorff multivalued system if and only if it is the inverse limit of an inverse system of multivalued systems of finite type (or graphs) satisfying the Mittag-Leffler condition. In particular, if the inverse system is countable and $+$directional, then the inverse limit is conjugate a single-valued continuous system.
\end{theorem*}

When we go through the literature, we find in \cite{BBK}, the authors asked a question:
\begin{ques}[\rm{\cite[Questions]{BBK}}]\label{ques1.1}
    Whether there exists a characterization of shadowing in terms of graph covers.
\end{ques}

In the aforementioned theorem, each constituent element of the inverse limit is a simple finite directed graph, which is a fundamental object in graph theory, as mentioned before, the space consists of the infinite walk of finite directed graph is a shift of finite type, so the method used in \cite{GM} can be transferred into multivalued version. Finally, by combining the thought appear in \cite{GM06, Sh14, GM}, we prove that 
\begin{theorem*}
    Let $f:X\to 2^X$ be a multivalued map on a compact totally disconnected Hausdorff space $X$. Then $f$ has shadowing property if and only if $(X,f)$ is conjugate to an inverse limit of multivalued of finite type such that the inverse system of their induced orbit systems satisfy Mittag-Leffler condition. 
\end{theorem*}

As an application and corollary, we answer the Question \ref{ques1.1}.

\begin{cor*}
    Let $\mathcal{G} = \{G_1 \xleftarrow{\varphi_1} G_2 \xleftarrow{\varphi_2} \cdots\}$ be a graph cover. The system has the shadowing property if and only if the inverse system of the orbit spaces (infinite walk) of $G_i=(V_i, E_i)$ ($i\in \mathbb{N}$) satisfies the Mittag-Leffler condition.
\end{cor*}

In the rest of the paper, we also show that shadowing property in mutivalued category is 'big', that is, it is dense in the class of upper semicontinuous maps on zero-dimensional space.

The paper is organized as follows. In section \ref{se2}, we develop the inverse limit of multivalued systems such that not only the multivalued structure is preserved but also the dynamical action is preserved. In section \ref{se3}, we introduce a useful tool extracted from \cite{GM06}, which help us to visualize the trajectory of a multivalued system. In section \ref{se4}, the relationship between multivalued shadowing property  and shift of finite type is preliminarily explored, and some interesting examples are also contained in this part. In section \ref{se5}, the graph cover method is applied to construct the structure of multivalued system on a zero-dimensional space. In section \ref{se6}, the structure of multivalued shadowing property  on zero-dimensional space is understood. In section \ref{se7}, we consider how many multivalued maps possess shadowing property in the class of all upper semicontinuous maps. In section \ref{se8}, we consider the shadowing property in general spaces and the openness is concerned.

\newpage
\section{Preliminary}\label{se2}
\subsection{Multivalued maps, orbits and graph walks}

\subsubsection{Multivalued maps}
In general, we denote by $\omega$ the set of natural numbers including $0$.
Let $X$ be a compact Hausdorff space (need not be metrizable). We denote by $2^X$ the class of all non-empty closed subsets of $X$ with the Vietoris topology. If $d$ is a metric on $X$, then $d$ induces a new metric named Hausdorff metric $d_H$ on $2^X$. It is known that $d_H$ is compatible with the Vietoris topology, and $d_H$ is given by 
\[
d_H(A,B) = \max\left\{\sup_{a \in A} d(a,B), \sup_{b \in B} d(A,b)\right\}\text{ for all }A,B\in 2^X.
\]

Recall that a map $f$ from $X$ to $2^X$ is said to be \textbf{upper semicontinuous} if for each $x\in X$ and an open neighborhood $U$ of $f(x)$, there is an open neighborhood $V$ of $x$ such that $f(y)\subset U$ for all $y\in V$; $f$ is said to be \textbf{lower semicontinuous} if for each $x\in X$ and an open set $U$ intersecting with $f(x)$, i.e., $f(x)\cap U\ne\emptyset$, there is an open neighborhood $V$ of $x$ such that $f(y)\cap U\ne\emptyset$ for all $y\in V$. Particularly, if $f$ is both upper and lower semicontinuous, then $f$ is continuous relative to the topologies on $X$ and $2^X$, respectively. 

Let $f:X\to 2^X$. Given $A\subset X$, we define the \textbf{inverse image} of $A$ under $f$ by 
\[f^{-1}(A) = \{ x \in X : f(x) \cap A \neq \emptyset \},\]
and write
\[f(A)=\{y\in f(x):x\in A\}=\bigcup_{a\in A}f(a).\]
In this paper, we call $f:X\to 2^X$ a \textbf{multivalued map} if $f$ is upper semicontinuous.

\begin{prop}\rm{(Folklore)}\label{prop2.1}
    Let $X$ be a compact Hausdorff space and $f:X\to 2^X$ a multivalued map. Then the following are equivalent:
    \begin{enumerate}
        \item $f$ is upper semicontinuous.
        \item For every closed set $A$ of $X$, the set $f^{-1}(A)$ is closed in $X$.
        \item \textbf{(Closed Graph Theorem)}. The graph
        \[
        G_f= \{x\times f(x) : x \in X\}
        \]
        is closed in $X \times X$.
    \end{enumerate}
    Besides, if $A$ is a compact set of $X$, then $f(A)$ is also compact.
\end{prop}

\begin{rem}
\begin{enumerate}
    \item If $f$ is a single-valued map, then $G_f$ is closed if and only if $f$ is continuous.
    \item By virtue of the closed graph theorem, within the category of graphs, we abandon the distinction between single-valued and multi-valued functions. Therefore, we uniformly use the symbol $f$ to denote a multivalued mapping.
\end{enumerate}
   
\end{rem}

\subsubsection{Orbit spaces and graph walks}
Let $X$ be a compact Hausdorff space. $\prod_{i\in\omega}X$ denotes the infinite product space of $X$ with the product topology. We denote $\{x_i\}_{i\in\omega}$ (or $\{x_i\}$ for short) as an element in $\prod_{i\in\omega}X$, and define
\[\Pi_k:\prod_{i\in\omega}X\to X,\,\{x_i\}\mapsto x_k,\quad\text{for all }\{x_i\}\in \prod_{i\in\omega}X\] as the $k$-th projection map.
 Throughout the paper, bold symbols such as $\boldsymbol{x},\boldsymbol{w},\boldsymbol{y}$ will sometimes be used to denote a point in the product space.
 If $X$ is metrizable with metric $d$, we define
\[\rho(\boldsymbol{x},\boldsymbol{y})=\sum_{i\in\omega}\frac{d(x_i,y_i)}{2^{i+1}},\quad\text{for all }\boldsymbol{x}=\{x_i\},\boldsymbol{y}=\{y_i\}\in \prod_{i\in \omega}X.\]
Then $\rho$ is a compatible metric on $\prod_{i\in \omega}X$.

Let $X$ be a compact Hausdorff space and $f:X\to 2^X$ be a multivalued map. We consider the action of a point $x\in X$ under iteration of $f$, viewing the pair $(X,f)$ as a dynamical system. The \textbf{orbit space} of $(X,f)$ is defined as
\[
{\rm{Orb}}_f(X) = \left\{ \{x_i\} \in \prod_{i\in \omega} X : x_{i+1} \in f(x_i) \text{ for all } i \in \omega \right\}.
\]
For clarity, an element of ${\rm{Orb}}_f(X)$ will be denoted by $(x_i)_{i\in\omega}$, or simply $(x_i)$.

The \textbf{(left) shift} transformation on ${\rm{Orb}}_f(X)$ is given by
\[
\sigma_f(\{x_i\})=\{y_i\}, \quad \text{where } y_i=x_{i+1} \text{ for all } i\in \omega.
\]
By the closed graph theorem for upper semicontinuous maps, ${\rm{Orb}}_f(X)$ is a closed subset of $\prod_{i\in\omega}X$, and consequently $({\rm{Orb}}_f(X),\sigma_f)$ constitutes a topological dynamical system in the classical sense.

\begin{rem}
    In particular, if $X$ is a finite set, ${\rm{Orb}}_f(X)$ is repesented as a space of infinite \textbf{graph walks}.
\end{rem}

\subsubsection{Semiconjugacy}
Given two single-valued dynamical systems $(X,f_1)$ and $(Y,f_2)$, recall that a continuous map $\pi$ between $X$ and $Y$ is said to \textbf{preserve dynamical action} if $\pi\circ f_1(x)=f_2\circ\pi(x)$ holds for all $x\in X$. Moreover, if $\pi$ is surjective, $(Y,f_2)$ is called a \textbf{factor} of $(X,f_1)$. Similarly, given two multivalued systems $(X,f)$ and $(Y,g)$, we consider the dynamical action preserving map, which is defined as follows. 

\begin{de}\label{def2.10}\rm{(\cite[cf. Definition 1.15]{Ak}).}
    Let $(X, f)$ and $(Y, g)$ be two multivalued systems and $\pi:X\to Y$ be a continuous map. We say that
    \begin{enumerate}
        \item  $\pi$ is a \textbf{semiconjugacy} from $(X,f)$ to $(Y,g)$ if $\pi(f(x))\subset g(\pi(x))$ for all $x\in X$.

        \item $\pi$ is an \textbf{exact semiconjugacy} from $(X,f)$ to $(Y,g)$ if $\pi(f(x))= g(\pi(x))$ for all $x\in X$.

        \item $\pi$ is a \textbf{supsemiconjugacy} from $(X,f)$ to $(Y,g)$ if $\pi(f(x))\supset g(\pi(x))$ for all $x\in X$.
    \end{enumerate}
    Particularly, if $\pi$ is also surjective such that conditions 1, 2 or 3 hold, we call $(Y,g)$ a \textbf{subfactor} (resp. \textbf{factor}, \textbf{supfactor}) of $(X,f)$.
\end{de}
The semiconjugacy defined in Definition \ref{def2.10} differs from the notion of semi-conjugacy introduced in \cite{KT}.
\begin{rem}
  To emphasize and simplify, if $\pi$ is a semiconjugacy from $(X,f)$ to $(Y,g)$ (not necessarily surjective), we say that $\pi$ satisfies \textbf{Condition \eqref{eq1.11}}, namely  

\begin{equation}\label{eq1.11}
    \pi(f(x)) \subset g(\pi(x)) \quad \text{for all } x \in X. 
\end{equation}

Moreover, it is interesting to observe that if the multivalued maps $f$ and $g$ are viewed as general graphs on topological spaces, then a semiconjugacy $\pi$ between $f$ and $g$ corresponds precisely to a \textbf{graph homomorphism} between the two graphs. Naturally, such a $\pi$ induces a \textbf{functorial} construction $\pi^*$ given by  

\[
\pi^* \colon \mathrm{Orb}_f(X) \to \mathrm{Orb}_g(Y), \quad \pi^*((x_i)) = (\pi(x_i)),
\]
which furthermore satisfies  
\[
\pi^* \circ \sigma_f((x_i)) = \sigma_g \circ \pi^*((x_i)) \quad \text{for all } (x_i) \in \mathrm{Orb}_f(X).
\]
\end{rem}

\begin{de}
    The system $(X, f)$ is said to be \textbf{conjugate} to $(Y, g)$ if there exists a homeomorphism $\pi: X \to Y$ such that
    \[
        f(x) = \pi^{-1}\big(g(\pi(x))\big) \quad \text{for all } x \in X.
    \]
\end{de}

If $f$ is a continuous single-valued map, then its orbit system $(\mathrm{Orb}_f(X),\sigma_f)$ is conjugate to $(X,f)$ via the 0-th projection map $\Pi_0$.

\subsection{Inverse systems and Mittag-Leffler condition}

\begin{de}
    Let $(\Lambda, \leq)$ be a directed set. An \textbf{inverse system} $(g^\eta_\lambda, X_\lambda)$ over $\Lambda$ consists of a collection of compact Hausdorff spaces $\{X_\lambda \mid \lambda \in \Lambda\}$, together with, for each pair $\lambda \leq \eta$, a continuous map $g^\eta_\lambda : X_\eta \to X_\lambda$ such that:
    \begin{itemize}
        \item $g^\lambda_\lambda = \mathrm{id}_{X_\lambda}$ for every $\lambda \in \Lambda$;
        \item $g^\nu_\lambda = g^\eta_\lambda \circ g^\nu_\eta$ whenever $\lambda \leq \eta \leq \nu$.
    \end{itemize}
    The \textbf{inverse limit} of the inverse system is defined as
    \[
        \varprojlim \{ g^\eta_\lambda, X_\lambda \} 
        = \bigl\{ (x_\lambda) \in \prod_{\lambda \in \Lambda} X_\lambda \;\bigm|\; x_\lambda = g^\eta_\lambda(x_\eta) \text{ for all } \lambda \leq \eta \bigr\},
    \]
    endowed with the topology inherited from the product topology on $\prod_{\lambda \in \Lambda} X_\lambda$.
\end{de}

\begin{de}
    An inverse system $(g^\eta_\lambda, X_\lambda)$ is said to satisfy the \textbf{Mittag-Leffler condition} (sometimes, ML for short) if for every $\lambda \in \Lambda$, there exists $\gamma \geq \lambda$ such that for each $\eta \geq \gamma$, 
    \[
    g^\gamma_\lambda(X_\gamma) = g^\eta_\lambda(X_\eta).
    \]
\end{de}

\begin{de}
    Let $(g^\eta_\lambda,X_\lambda)$ be an inverse system. For each $\lambda\in\Lambda$, the \textbf{stable image} of $X_\lambda$ is defined as
    \[
    X_\lambda' = \bigcap_{\lambda \leq \eta} g^\eta_\lambda(X_\eta).
    \]
\end{de}

Suppose the inverse system $(X_\lambda, g^\eta_\lambda)$ satisfies the Mittag-Leffler condition. Fix $\lambda \in \Lambda$ and let $\eta \geq \lambda$ be an index witnessing this condition. Define the stable image by $X'_\lambda = g^\eta_\lambda(X_\eta)$. It follows that $X'_\lambda = g^\gamma_\lambda(X_\gamma)$ for all $\gamma \geq \eta$. Moreover, for any $\lambda_1 \geq \lambda$, the restriction map $g^{\lambda_1}_\lambda|_{X'_{\lambda_1}}$ is surjective onto $X'_\lambda$. Consequently, the system of stable images induces the same inverse limit as the original one.

Hence, we have the following useful property $(*)$ (see \cite{GM}):
If an inverse system $(X_\lambda, g^\eta_\lambda)$ satisfies the Mittag-Leffler condition, and $\gamma$ witnesses this condition for $\mu$, then for any $x \in g^\gamma_\mu(X_\gamma) \subseteq X_\mu$, we have
\[
    \pi^{-1}_\mu(x) \cap \varprojlim \{ g^\eta_\lambda, X_\lambda \} \neq \emptyset.
\]

\subsection{Inverse limit of multivalued systems}

\begin{de}\label{def1.4}
    Let $\{(X_\lambda, f_\lambda)\}_{\lambda \in \Lambda}$ be a collection of multivalued systems indexed by a directed set $(\Lambda, \leq)$. Suppose that for each pair $\lambda \leq \eta$ in $\Lambda$, there exists a continuous map $g^\eta_\lambda : X_\eta \to X_\lambda$. The collection $\mathcal{S} = \{ g^\eta_\lambda, (X_\lambda, f_\lambda) \}$ is called an \textbf{inverse system} if the following conditions hold:
    \begin{itemize}
        \item for any $\lambda \leq \eta$, the map $g^\eta_\lambda : (X_\eta, f_\eta) \to (X_\lambda, f_\lambda)$ satisfies Condition \eqref{eq1.11};
        \item the inverse system $(g_\lambda^\eta, X_\lambda)$ satisfies the Mittag-Leffler condition;
        \item for any $\lambda \leq \eta \leq \nu$ and any $x \in X_\nu$, $g^\nu_\lambda(x) = g^\eta_\lambda \circ g^\nu_\eta(x)$.
    \end{itemize}
    The \textbf{inverse limit multivalued system} of $\mathcal{S}$ is defined as the pair $\left(\varprojlim X_\lambda, \varprojlim f_\lambda\right)$, where the space $\varprojlim X_\lambda = \varprojlim \{ g^\eta_\lambda, X_\lambda \}$ is the standard inverse limit space, and the map $\varprojlim f_\lambda$ is defined by
    \[
        \varprojlim f_\lambda\left((x_\lambda)\right) = \left\{ (y_\lambda) \in \varprojlim X_\lambda : (x_\lambda, y_\lambda) \in G_{f_\lambda} \text{ for all } \lambda \in \Lambda \right\}.
    \]
\end{de}

\begin{rem}
    The Mittag-Leffler condition is not required if every $f_\lambda$ is a single-valued map on $X_\lambda$.
\end{rem}

\begin{thm}\label{lem1.12}
    The pair $\left( \varprojlim \{ g^\eta_\lambda, X_\lambda \}, \varprojlim f_\lambda \right)$ defined in Definition \ref{def1.4} is an upper semicontinuous multivalued system.
\end{thm}

\begin{proof}
    We divide the proof into three steps:
    \begin{enumerate}
        \item We prove that $\varprojlim f_\lambda((x_\lambda))$ is non-empty for each $(x_\lambda) \in \varprojlim \{ g^\eta_\lambda, X_\lambda \}$.
        \item For each $(x_\lambda) \in \varprojlim \{ g^\eta_\lambda, X_\lambda \}$, we verify that $\varprojlim f_\lambda((x_\lambda))$ is a closed set in $\prod_{\lambda\in \Lambda}X_\lambda$.
        \item We show that the graph of $\varprojlim f_\lambda$ is closed in $\left(\varprojlim \{ g^\eta_\lambda, X_\lambda \}\right)^2$.
    \end{enumerate}

    \textbf{Step 1.} If the inverse limit $\varprojlim \{ g^\eta_\lambda, X_\lambda \}$ is empty, the statement holds trivially. Let $(x_\lambda) \in \varprojlim \{ g^\eta_\lambda, X_\lambda \}$. Since $( g^\eta_\lambda, X_\lambda )$ satisfies the Mittag-Leffler condition, for each $\lambda \in \Lambda$, there exists some $\gamma \geq \lambda$ such that for any $\eta \geq \gamma$, we have
    \[
        g^\gamma_\lambda(X_\gamma) = g^\eta_\lambda(X_\eta).
    \]
    Fix $\lambda$, let $\gamma$ witness the above condition, and take $x_\gamma' \in f_\gamma(x_\gamma)$. Since $g^\gamma_\lambda: (X_\gamma, f_\gamma) \to (X_\lambda, f_\lambda)$ satisfies Condition \eqref{eq1.11}, we have $g^\gamma_\lambda(x_\gamma') = x_\lambda' \in f_\lambda(x_\lambda)$, and $x'_\lambda$ is in the stable image $X_\lambda'$. Therefore, there exists a point $\boldsymbol{w}_\lambda = (z_\mu)_{\mu \in \Lambda} \in \varprojlim \{ g^\eta_\lambda, X_\lambda \}$ such that 
    \[
        \pi_\lambda(\boldsymbol{w}_\lambda) = x_\lambda'.
    \]
    Consider $\nu \in \Lambda$ with $\lambda \leq \nu$. Following a similar discussion, there is an element $\boldsymbol{w}_\nu \in \varprojlim \{ g^\eta_\lambda, X_\lambda \}$ such that 
    \[
        \pi_\nu(\boldsymbol{w}_\nu) \in f_\nu(x_\nu) \quad \text{and} \quad \pi_\lambda(\boldsymbol{w}_\nu) = g^\nu_\lambda(\pi_\nu(\boldsymbol{w}_\nu)) \in f_\lambda(x_\lambda)
    \]
    (as $g^\nu_\lambda$ satisfies Condition \eqref{eq1.11}).
    
    Thus, by repeating the process over $\Lambda$, we can find a net $\{\boldsymbol{w}_\lambda : \lambda \in \Lambda\}$ in $\varprojlim \{ g^\eta_\lambda, X_\lambda \}$ such that:
    \[
        \pi_\nu(\boldsymbol{w}_\nu) \in f_\nu(x_\nu) \text{ for each } \nu \in \Lambda,
    \]
    and if $\nu \leq \eta$, since $g^\eta_\nu$ satisfies Condition \eqref{eq1.11}, then
    \[
        \pi_\nu(\boldsymbol{w}_\eta) \in f_\nu(x_\nu) \quad (\text{since } \pi_\nu(\boldsymbol{w}_\eta) = g^\eta_\nu(\pi_\eta(\boldsymbol{w}_\eta)) \subset g^\eta_\nu(f_\eta(x_\eta)) \subset f_\nu(x_\nu)).
    \]
    Since $\varprojlim \{ g^\eta_\lambda, X_\lambda \}$ is compact, the net $\{\boldsymbol{w}_\lambda : \lambda \in \Lambda\}$ has a cluster point $\boldsymbol{y} = (y_\lambda)$. From the product topology, we know that $\{\pi_\mu(\boldsymbol{w}_\lambda) : \lambda \in \Lambda\}$ is frequently in any neighborhood of $y_\mu$. From the discussion above, $\{\pi_\mu(\boldsymbol{w}_\lambda) : \lambda \in \Lambda\}$ is eventually in $f_\mu(x_\mu)$ for each $\mu \in \Lambda$. Thus, since $f_\mu(x_\mu)$ is closed, $y_\mu \in f_\mu(x_\mu)$ for each $\mu \in \Lambda$. Hence, $\boldsymbol{y} \in \prod_{\lambda \in \Lambda} f_\lambda(x_\lambda)$. Therefore, $\varprojlim f_\lambda((x_\lambda))$ is non-empty for $(x_\lambda) \in \varprojlim \{ g^\eta_\lambda, X_\lambda \}$.

    \textbf{Step 2.} Let $(x_\lambda) \in \varprojlim \{ g^\eta_\lambda, X_\lambda \}$. Suppose $\{\boldsymbol{w}_\xi\}$ is a net in $\varprojlim f_\lambda((x_\lambda))$ converging to some point $\boldsymbol{w} \in \varprojlim \{ g^\eta_\lambda, X_\lambda \}$. Then for each $\lambda \in \Lambda$, $\pi_\lambda(\boldsymbol{w}_\xi)$ converges to $\pi_\lambda(\boldsymbol{w})$. As $f_\lambda$ is upper semicontinuous (and has closed values), $\pi_\lambda(\boldsymbol{w}) \in f_\lambda(x_\lambda)$. Therefore, $\boldsymbol{w} \in \varprojlim f_\lambda((x_\lambda))$, implying the set is closed.

    \textbf{Step 3.} Suppose a net $(\boldsymbol{x}_\alpha, \boldsymbol{y}_\alpha)$ in the graph $G_{\varprojlim f_\lambda}$ converges to some point $(\boldsymbol{x}, \boldsymbol{y})$. Since $\varprojlim \{ g^\eta_\lambda, X_\lambda \}$ is closed in the product space, $(\boldsymbol{x}, \boldsymbol{y}) \in \left(\varprojlim \{ g^\eta_\lambda, X_\lambda \}\right)^2$. For each $\mu \in \Lambda$, $f_\mu$ is an upper semicontinuous map, so $(\pi_\mu(\boldsymbol{x}_\alpha), \pi_\mu(\boldsymbol{y}_\alpha))$ converges to $(\pi_\mu(\boldsymbol{x}), \pi_\mu(\boldsymbol{y})) \in G_{f_\mu}$. Therefore, $\pi_\mu(\boldsymbol{y}) \in f_\mu(\pi_\mu(\boldsymbol{x}))$ for all $\mu$, which implies $(\boldsymbol{x}, \boldsymbol{y}) \in G_{\varprojlim f_\lambda}$. Hence, $\varprojlim f_\lambda$ is upper semicontinuous by the Closed Graph Theorem.
\end{proof}

\section{Expansiveness}\label{se3}

In this section, we present an equivalent characterization of expansiveness for multivalued systems. Furthermore, we use open covers to describe the itineraries of points.

\subsection{Tracking orbits of multivalued maps}\label{se:3.1}

Let $(X,f)$ be a multivalued system and let $x_0 \in X$. Consider the iteration of $x_0$ under $f$, the point $x_0$ may produce different orbits. A natural question arises: how can we observe or record one specific orbit? We find that open covers of $X$ provide a method to address this, as follows.

First, we denote by $\mathcal{FOC}(X)$ the class of all finite open covers of $X$. Let $\mathcal{U} \in \mathcal{FOC}(X)$. We refer to it as $\mathcal{U}$-\textbf{pattern-space}. 

Given an initial point $x_0 \in X$, suppose $x_0 \in U_0$ at time $0$, and $x_0$ moves to some $x_1 \in U_1 \in \mathcal{U}$ at time $1$ (denoted as $x_0 \overset{f}{\curvearrowright} x_1$). Then we have:
\[
    x_0 \in U_0, \quad f(x_0) \cap U_1 \neq \emptyset \implies x_0 \in U_0 \cap f^{-1}(U_1).
\]
Continuing to observe $x_1$, if $x_1$ moves to some $x_2 \in U_2 \in \mathcal{U}$ at time $2$ (i.e., $x_0 \overset{f}{\curvearrowright} x_1 \overset{f}{\curvearrowright} x_2$), then:
\[
    x_1 \in U_1, \quad f(x_1) \cap U_2 \neq \emptyset \implies x_1 \in U_1 \cap f^{-1}(U_2),
\]
plus
\[
    x_0 \in U_0 \cap f^{-1}(U_1 \cap f^{-1}(U_2)).
\]

Repeating this process indefinitely, we obtain a method to track the orbit of $x_0$. Specifically, if $(x_i)_{i\geq 0}$ is an $f$-orbit of $x_0$ such that $x_i \in U_i \in \mathcal{U}$ for each $i$, then for every $i \geq 1$, we have
\[
    x_0 \in U_0 \cap f^{-1}(U_1 \cap f^{-1}(U_2 \cap \cdots \cap f^{-1}(U_i))).
\]
More generally, for any $j \geq i$,
\[
    x_i \in U_i \cap f^{-1}(U_{i+1} \cap f^{-1}(U_{i+2} \cap \cdots \cap f^{-1}(U_j))).
\]

\begin{figure}[htbp]\label{F1}
    \centering 
    
    \begin{tikzpicture}[
        orbit_set/.style={draw, ellipse, minimum width=2.2cm, minimum height=1.6cm, thick},
        sub_region/.style={draw, ellipse, dashed, minimum width=1.2cm, minimum height=0.8cm, pattern=north east lines, pattern color=gray!60},
        point/.style={circle, fill=black, inner sep=1.5pt},
        map_arrow/.style={->, >=Stealth, thick, bend left=35}
    ]

        \node[orbit_set, label=below:$\mathcal{U} \ni U_0$] (U0) at (0,0) {};
        \node[orbit_set, label=below:$\mathcal{U} \ni U_1$, right=2.5cm of U0] (U1) {};
        \node[orbit_set, label=below:$\mathcal{U} \ni U_2$, right=2.5cm of U1] (U2) {};

        \node[sub_region] (invU1) at (U0.center) {};
        \node[font=\tiny, above right] at (invU1.north east) {$U_0 \cap f^{-1}(U_1)$};

        \node[sub_region] (invU2) at (U1.center) {};
        \node[font=\tiny, above right] at (invU2.north east) {$U_1 \cap f^{-1}(U_2)$};

        \node[point, label=below left:$x_0$] (x0) at (invU1.center) {};
        \node[point, label=below left:$x_1$] (x1) at (invU2.center) {};
        \node[point, label=below left:$x_2$] (x2) at (U2.center) {};

        \draw[map_arrow] (x0) to node[above, font=\small] {$f$} (x1);
        \draw[map_arrow] (x1) to node[above, font=\small] {$f$} (x2);

        \node[font=\small\bfseries, above=1cm of U0] {Time $0$};
        \node[font=\small\bfseries, above=1cm of U1] {Time $1$};
        \node[font=\small\bfseries, above=1cm of U2] {Time $2$};

    \end{tikzpicture}
    
    \caption{Record orbit} 
    \label{fig:orbit_schematic} 
\end{figure}
Based on the intuitive tracking procedure described above, we now formalize these concepts into mathematical definitions. We introduce the notion of the \textit{orbital discriminant}, that is,

\begin{de}[Orbital Discriminant]
    Let $(X, f)$ be a set-valued system and let $\mathcal{U} \in \FOC(X)$ be a finite open cover.
    \begin{enumerate}
        \item[(1)]  For a finite sequence of open sets $\mathbf{U}_n = \{U_0, U_1, \dots, U_n\} \subseteq \mathcal{U}$, we define the \textbf{$n$-tuple orbital discriminant} as:
        \[
            [U_0, \dots, U_n] := U_0 \cap f^{-1}\left( U_1 \cap f^{-1}\left( U_2 \cap \cdots \cap f^{-1}(U_n) \right) \right).
        \]
        This set consists of all points $x_0 \in U_0$ for which there exists a partial orbit $x_0 \overset{f}{\curvearrowright} \cdots \overset{f}{\curvearrowright} x_n$ such that $x_i \in U_i$ for all $0 \le i \le n$.
        
        \item[(2)]  For an infinite sequence $\mathbf{U} = \{U_i\}_{i \geq 0} \subseteq \mathcal{U}$, we define the \textbf{orbital discriminant} by:
        \[
            [\mathbf{U}] := \bigcap_{n=0}^{\infty} [U_0, \dots, U_n] = U_0 \cap f^{-1}\left( U_1 \cap f^{-1}\left( U_2 \cap \cdots \right) \right).
        \]
        This denotes the set of points in $U_0$ whose forward orbit can be realized within the sequence $\{U_i\}_{i \geq 0}$. Sometimes, we also write $[U_i]$.
    \end{enumerate}
\end{de}

\begin{de}[$\mathcal{U}$-orbit-pattern]
    Let $(X, f)$ be a multivalued system and $\mathcal{U} \in \FOC(X)$. An infinite sequence $\mathbf{U} = \{U_i\}_{i \geq 0}$ of elements in $\mathcal{U}$ is called a \textbf{$\mathcal{U}$-orbit-pattern} if its orbital discriminant is non-empty, i.e.,
    \[
        [\mathbf{U}] \neq \emptyset.
    \]
    Geometrically, this implies that there exists at least one actual orbit $(x_i)_{i \geq 0}$ of the system such that $x_i \in U_i$ for all $i \geq 0$.
\end{de}

\begin{rem}\label{rem3.3}
\begin{enumerate}
    \item The operator $[\cdot]$ provides a method to characterize the orbit $(x_i)_{i \in \omega}$ through the lens of the open cover $\mathcal{U}$. Specifically, for any infinite sequence $\mathbf{U} = \{U_i\}_{i \in \omega} \subseteq \mathcal{U}$, the orbital discriminant is the limit of the finite tuple discriminants:
    \[
        [\mathbf{U}] = \bigcap_{n \in \omega} [U_0, \dots, U_n].
    \]
    \item The method above is extracted from \cite{GM06}, in which the finite orbit of a point in some minimal Cantor system is observed via clopen partitions, and the structure of such minimal Cantor system is described as the inverse limit of some finite directed graphs.
\end{enumerate}
\end{rem}

\begin{rem}
    Consider that $f$ is a single-valued map. As the distributivity of the preimage over intersections (i.e., $f^{-1}(A \cap B) = f^{-1}(A) \cap f^{-1}(B)$), the nested structure of the orbital discriminant is presented as follows:
    \[
        [U_0, \dots, U_n] = U_0 \cap f^{-1}(U_1) \cap f^{-2}(U_2) \cap \cdots \cap f^{-n}(U_n) = \bigcap_{k=0}^n f^{-k}(U_k).
    \]
    Thus, our definition naturally generalizes the classical notion of a \textbf{cylinder set} in symbolic dynamics.
\end{rem}

\subsection{Expansiveness and generators for multivalued maps}

In this subsection, we demonstrate the applicability of the framework established above. we introduce the concepts of {generators} in the context of multivalued maps. These definitions generalize the standard single-valued counterparts found in \cite{KR} and \cite{PW}.

\begin{de}
    Let $(X, f)$ be a multivalued system and let $\mathcal{U} \in \mathcal{FOC}(X)$.
    \begin{enumerate}
        \item $\mathcal{U}$ is called a \textbf{generator} for $f$ if for every sequence $\mathbf{U}=\{U_i\}_{i\geq 0}$ in $\mathcal{U}$, $[\overline{\mathbf{U}}]$ contains at most one point, where $\overline{\mathbf{U}} := \{\overline{U}_i\}_{i\geq 0}.$
        
        \item $\mathcal{U}$ is called a \textbf{weak generator} for $F$ if for every sequence $\mathbf{U}=\{U_i\}_{i\geq 0}$ in $\mathcal{U}$,  $[\mathbf{U}]$ contains at most one point.
    \end{enumerate}
\end{de}

\begin{prop}\rm{(cf. \cite[Theorem 5.23]{PW}).}\label{prop3.6}
    Let $f: X \to 2^X$ be a multivalued map on a compact Hausdorff space $X$, and let $\mathcal{U}$ be a generator for $f$. Then, for every open cover $\mathcal{V}$ of $X$, there exists $N \in \mathbb{N}$ such that for any sequence $U_0, \dots, U_N \in \mathcal{U}$, the set
       $ [U_0, \dots, U_N]$
    is contained in some element of $\mathcal{V}$.
\end{prop}
\begin{proof}
   
    Let $\alpha$ be an element of the uniformity (an entourage) which refines the cover $\mathcal{V}$ (essentially a Lebesgue entourage). We show there exists $N > 0$ such that for any orbit segment of length $N$, the set $[U_0, \dots, U_N]$ is $\alpha$-small (i.e., contained in $\mathcal{V}$).

    Suppose the contrary. Then for each $n$, there exist points $x_n, y_n \in [U_0^{(n)}, \dots, U_n^{(n)}]$ such that $(x_n, y_n) \notin \alpha$.
    By compactness, passing to a subsequence, $x_{n_k} \to x$ and $y_{n_k} \to y$. Since the complement of $\alpha$ is closed, we have $(x, y) \notin \alpha$, which implies $x \neq y$.

    Since $\mathcal{U}$ is finite, there exists $U_0\in \mathcal{U}$ such that for infinitely many $j_k$, we have $x_{j_k},y_{j_k}\in \overline{U_0}$.  Similarly, we can find $U_1\in \mathcal{U}$ such that for infinitely many $j_k$, $x_{j_k},y_{j_k}\in U_0$ and $f(x_{j_k})\cap \overline{U_1}\ne\emptyset$, $f(y_{j_k})\cap \overline{U_1}\ne\emptyset$. As $f$ is upper semicontinuous, we obtain $x,y\in f^{-1}(\overline{U_1})$. 
   By induction, for each $n$, we can find $U_n\in \mathcal{U}$ such that $x_{j_k},y_{j_k}\in [\overline{U_0},\cdots,\overline{U_n}]$ for infinitely many $j_k$. Thus, we get $x\ne y\in \bigcap_{i\in \omega
   }[\overline{U}_0,\cdots,\overline{U}_i]$, but this contracts that $\mathcal{U}$ is a generator. Thus, there is a $N$ such that the diameter of $[U_0,\cdots,U_N]$ is less than $\alpha$, 

\end{proof}

\begin{de}[cf. \cite{Wi}]
    A multivalued map $f: X \to 2^X$ is said to be \textbf{expansive} if there exists a constant $\delta > 0$ such that $x = y$ whenever there exist $f$-orbits $(x_i)_{i\geq 0}$ and $(y_i)_{i\geq 0}$ of $x$ and $y$, respectively, satisfying
       $ d(x_i, y_i) \le \delta \quad \text{for all } i \ge 0.$
\end{de}

With Proposition \ref{prop3.6}, one can easily prove that 
\begin{thm}
    A multivalued system $(X,f)$ is expansive if and only if $f$ has a generator if and only if $f$ has weak generator.
\end{thm}

Recall a well-known result for expansive single-valued maps.

\begin{thm}[cf. \cite{KR, PW}]
    Let $(X,f)$ be an expansive single-valued system. $(\operatorname{Orb}_f(X), \sigma_f)$ is a factor of a subshift over a finite alphabet. Moreover, if $X$ is totally disconnected, then $({\rm{Orb}}_f,\sigma_f)$ is conjugate to a subshift.
\end{thm}

\begin{rem}\label{rmk3.10}
Interestingly, for multivalued systems, the examples in the next subsection demonstrates that the orbit space of an expansive system can be conjugate to a subshift over a finite alphabet, even when the phase space is a continuum.    
\end{rem}

\section{Shifts of finite type and multivalued shadowing property}\label{se4}
in this section, we consider the connection of shift of finite type and multivalued shadowing property. As far as we know, there is little discussions for multivalued version. In this part, some interesting examples are constructed and theorems are established.

\begin{de}\rm{(cf. \cite{MMT}).}
Let $f:X\to 2^X$ be an upper semicontinuous map. We say that $f$ has \textbf{shadowing property} if for each $\varepsilon>0$, there is $\delta>0$ such that if $\{x_i\}_{i\geq 0}$ is a sequence with $d(f(x_i),x_{i+1})<\delta$, where $d(f(x_i),x_{i+1}):=\inf\{d(x,x_{i+1}):x\in f(x_i)\}$, then there is an $f$-orbit $(y_i)_{i\geq 0}\in {\rm{Orb}}_f(X)$ such that $d(x_i,y_i)\leq \varepsilon$ for all $i\geq 0$.
\end{de}

If $\{x_i\}_{i\geq 0}$ satisfies $d(f(x_i),x_{i+1})<\delta$, we say that $\{x_i\}_{i\geq 0}$ is a $\delta$-\textit{pseudo orbit} of $f$, and if $\{y_i\}_{i\geq 0}$ satisfies $d(x_i,y_i)\leq \varepsilon$, we say that $\{y_i\}_{i\geq 0}$  $\varepsilon$-\textit{shadows} $\{x_i\}_{i\geq 0}.$

Here, we give an observation for multivalued maps.

\begin{lem}[\cite{CKY}, cf. Lemma 2.1 for single-valued maps]
    Let $X$ be a compact metric space. Let $f, g \colon X \to 2^X$ be multivalued maps and $\varepsilon, \delta > 0$. 
    Suppose that for all $x \in X$,
       $ \overline{B}_d(f(x), \varepsilon+\delta) \subset g(\overline{B}_d(x, \varepsilon)),$
    where $\overline{B}_d(x, r) := \{y \in X : d(x, y) \le r\}$ denotes the closed ball. 
    Then, each $\delta$-pseudo orbit $\{x_i\}_{i\geq 0}$ of $f$ can be $\varepsilon$-shadowed by a $g$-orbit $\{z_i\}_{i\geq 0}$.
\end{lem}

\begin{proof}
    Let $\{x_i\}_{i\geq 0}$ be a $\delta$-pseudo orbit of $f$. We define a sequence of sets $\mathbf{W}_0, \dots, \mathbf{W}_n, \dots$ as follows:
    \[ 
        \mathbf{W}_i := [\overline{B}_d(x_0,\varepsilon),\cdots, \overline{B}_d(x_i,\varepsilon)] = \overline{B}_d(x_0,\varepsilon) \cap g^{-1}(\overline{B}_d(x_1,\varepsilon) \cap \cdots \cap g^{-1}(\overline{B}_d(x_i,\varepsilon))).
    \]
 $\mathbf{W}_i$ is closed, as $g$ is upper semicontinuous.
 
 Note that for a multivalued map $g$, we have the following property:
    \begin{equation}\label{eq4.1}
        g(A)\cap B \subset g(A\cap g^{-1}(B)).
    \end{equation}
    Therefore, we have the following inclusion:
    \begin{align*}
        g^i(\mathbf{W}_i) &\supset g^{i-1}(g(\overline{B}_d(x_0,\varepsilon))\cap [\overline{B}_d(x_1,\varepsilon),\cdots,\overline{B}_d(x_i,\varepsilon)])\\
        &\supset g^{i-1}(\overline{B}_d(f(x_0),\delta+\varepsilon)\cap [\overline{B}_d(x_1,\varepsilon),\cdots,\overline{B}_d(x_i,\varepsilon)])\\
        &\supset g^{i-1}([\overline{B}_d(x_1,\varepsilon),\cdots,\overline{B}_d(x_i,\varepsilon)])\\
        &\supset \cdots \supset \overline{B}_d(x_i,\varepsilon),
    \end{align*}
  which means that $\mathbf{W}_i\ne\emptyset$ for all $i\geq 0$.  Thus, by the finite intersection property of compact sets, $\bigcap_{i} \mathbf{W}_i \neq \emptyset$. For any $z_0 \in \bigcap_{i} \mathbf{W}_i$, there exists a $g$-orbit $\{z_i\}_{i\geq 0}$ with $z_i \in \overline{B}_d(x_i,\varepsilon)$, which means it $\varepsilon$-shadows $\{x_i\}$.
\end{proof}

\begin{cor}\label{cor2.3}
For any $\varepsilon>0$ if there is $\delta>0$ such that $\overline{B}_d(f(x),\delta+\varepsilon)\subset f(\overline{B}_d(x,\varepsilon))$ for all $x\in X$, then $f:X\to 2^X$ has shadowing property.
\end{cor}

\begin{exam}
    Let $X=[0,1]$ be the unit interval. We define
    \[
f=\begin{cases}
    2x & \text{ if }x\in [0,\frac{1}{2}),\\
    \{0,1\}&\text{ if }x=\frac{1}{2},\\
    2x-1 &\text{ if }x\in (\frac{1}{2},1].
\end{cases}
    \]
For each $x\in [0,1]$, we consider its binary expansion, like 
\[x_0=\Sigma_{i=1}^\infty1/2^ia_i,\text{ where }a_i\in \{0,1\}.\]
Then we have:

\textbf{Fact 1.} Every number in $[0,1]$ has at most two representation;

\textbf{Fact 2.} apart from rational numbers with bases that are powers of $2$, all other numbers possess only one binary representation.

Let $(\{0,1\}^{\mathbb{N}},\sigma)$ be the full shift over alphabet $\{0,1\}$, we define a map
\[\pi:(\{0,1\}^{\mathbb{N}},\sigma)\to (\mathrm{Orb}_f,\sigma_f)\]
by  \[(a_i)_{i\geq 0}\mapsto \left(\sum_{i=0}^\infty 1/2^ia_i, \sum_{i=1}^\infty 1/2^ia_i,\cdots, \sum_{i=n}^\infty 1/2^ia_i,\cdots\right).\]
It is a standard way to prove that it is a hmomeomorphsim and a dynamical conjugacy, which means that the action of $([0,1],f)$ is actually an action of  flipping coin.
\end{exam}

\begin{rem}
    The example above is expansive, which corresponds to Remark \ref{rmk3.10}. Besides, it  satisfies the condition of Corollary \ref{cor2.3}, thus, it has multivalued shadowing property.
\end{rem}

 The example above shows that a orbit space of an multivalued system on continuum may fail to be connected, here, we include a multivalued map such that its phase space is still $[0,1]$, but its orbit is the Golden Mean Shift (a strictly subshift of finite type).

\begin{exam}
    Let $X=[0,1]$. The map $f:X\to 2^X$ is defined by 
    \begin{equation*}
f(x) = \begin{cases} 
\{\frac{5}{3}x\} & x \in [0, 0.6) \quad (I_0-0.6) \\
\{0, 1\} & x = 0.6 \\
\{\frac{3}{2}(x - 0.6)\} & x \in (0.6, 1] \quad (I_1)
\end{cases}
\end{equation*}
We shall prove that its orbit is zero-dimensional and conjuagate to the  golden mean shift (forbidding the word $"11"$).

Let $I_0=[0,0.6]$ and $I_1=( 0.6,1]$. Note that $f(I_0)\cap I_1\ne\emptyset$, $f(I_0)\cap I_1\ne\emptyset$ and $f(I_1)\cap I_0\ne\emptyset$, but $f(I_1)\cap I_1=\emptyset$. Actually, if $x\in I_1$, then $f(x)\in I_0$. We claim that ${\rm{Orb}}_f$ is conjugate to the golden mean shift $\Sigma$. Here, we employ the orbit tracking method described in Section~\ref{se:3.1}.
\begin{figure}[htbp]
    \centering
    \begin{tikzpicture}[shorten >=1pt, node distance=3cm, on grid, auto]
      
        \tikzstyle{every state}=[fill=white, draw=black, text=black, minimum size=1cm]

        \node[state] (0)              {$I_0$};
        \node[state] (1) [right=of 0] {$I_1$};

        \path[->, >=stealth, thick]
            (0) edge [loop above] node {} (0)
                edge [bend left=20]  node {} (1)
            (1) edge [bend left=20]  node {} (0);
    \end{tikzpicture}
    \caption{The finite directed graph representation of golden mean shift}
    \label{fig:golden_mean_shift}
\end{figure}

For any $\boldsymbol{x}\in {\rm{Orb}}_f$, we define 
\[\varphi:\boldsymbol{x}\mapsto (a_i)_{i\geq 0}\in \Sigma, \text{ if }x_i\in I_{a_i}.\]
First, we prove that $\varphi$ is injective. Suppose $\varphi(\boldsymbol{x})=\varphi(\boldsymbol{y})$ and $x_i\ne y_i$ with $\left\vert x_i-y_i\right\vert>\eta>0$ for some $i\geq 0$, but $x_s, y_s$, $s\geq i$, in the same $I_{a_s}$, then 
\[\left\vert f(x_i)-f(y_i)\right\vert\geq \min\left\{\frac{5}{3}\left\vert x_i-y_i \right\vert,\frac{3}{2}\left\vert x_i-y_i\right\vert\right\}\geq \frac{5}{3}\eta,\]
After applying $f$ to $x_i, y_i$ $n$ times, we can obtain
\[\left\vert f^n(x_i)-f^n(y_i)\right\vert\geq \left(\frac{5}{3}\right)^n\eta>0.6,\]
which is impossible. Hence, $x_i=y_i$ for all $i\geq 0$.

Thus, $\varphi:{\rm{Orb}}_f([0,1])\to \Sigma$ is injective.

Next, we claim that $\varphi$ is surjective. It suffices to prove that for any finite word $w=a_0\cdots a_n$ in $\Sigma$, the set 
\[\mathbf{I}=I_{a_0}\cap f^{-1}(I_{a_1}\cap \cdots f^{-1}(I_{a_n}))\ne\emptyset.\]

Recall that we have 
\begin{itemize}
    \item $f(I_0-0)=[0,1]-0=I_0\cup I_1-0$;
    \item $f(I_1)=I_0-0$.
\end{itemize}
Thus, 
\begin{align*}
    f^n(\mathbf{I})&\supset f^{n-1}(f(I_{a_0})\cap I_{a_1}\cap f^{-1}(I_{a_2}\cdots f^{-1}(I_{a_n}))),
\end{align*}
as $I_1I_1$ can not appear, one can use the property $f(A\cap f^{-1}(B))\supset f(A)\cap B$ by induction,  we can obtain $f^n(\mathbf{I})\ne\emptyset$. Therefore, $\mathbf{I}\ne\emptyset$.

Above all, its orbit is conjugate to the golden mean shift.
\end{exam}

\begin{rem}
    If we consider the finite shadowing property (see e.g., \cite{DGS, MMT}), then the following modified example shows that 
                  \begin{align}\label{eqe4.1}
f(x)=\left\{
             \begin{aligned}
             &\{2x\}, &0\leq x<1,  \\  
             &\{0\}, & x=1,\\  
             &\left\{2x-2\right\}, &1<x\leq 2.
             \end{aligned}  
\right.
\end{align}
a non-continuous map and its graph is not closed. $f$ has finite shadowing property but not shadowing property.
\end{rem}

As we shall consider the multivalued shadowing property on totally disconnected space, we include here a totally disconnected multivalued system. Later, we shall see shadowing property is a common phenomenon that exists almost universally on compact totally disconnected spaces.
\begin{exam}
    Let $C_1, C_2,$ and $C_3$ be the Cantor ternary sets generated on the intervals $[0,1]$, $[1,2]$, and $[2,3]$, respectively. Let $X = C_1 \cup C_2 \cup C_3$ be equipped with the Euclidean metric. Each element $x \in X$ admits a ternary expansion of the form:
    \[
        x = a_0 + \sum_{i=1}^\infty a_i 3^{-i}, \quad \text{where } a_0 \in \{0,1,2\} \text{ and } a_i \in \{0,2\} \text{ for } i \ge 1.
    \]
    Define the map $f: X \to 2^X$ by:
    \[
    f(x) = \begin{cases}
        \{3x\}, & x \in C_1 \setminus \{1\}, \\
        \{0, 3\}, & x \in \{1, 2\}, \\
        \{3x-3\}, & x \in C_2 \setminus \{1, 2\}, \\
        \{3x-6\}, & x \in C_3 \setminus \{2\}.
    \end{cases}
    \]
    By Corollary \ref{cor2.3}, $f$ is an upper semicontinuous map with the shadowing property on the compact totally disconnected Hausdorff space $X$.
\end{exam}

\subsection{Shifts of finite type and multivalued shadowing property}
 To begin with, we recall some notions about shift spaces on finite alphabets.
\begin{de}
    Let $S$ be a non-empty finite set and provide it with discrete topology. The \textit{full shift} on the alphabet $S$ is the set
    \[
        S^\omega = \prod_{i\in \omega} S = \{ (x_i)_{i\in\omega} : x_i \in S \text{ for all } i \in \omega \},
    \]
    endowed with the product topology. The standard metric defined on $S^\omega$ is given by
    \[
        d_\infty(\boldsymbol{x}, \boldsymbol{y}) = \begin{cases}
            0, & \text{if } x_i = y_i \text{ for all } i \in \omega, \\
            2^{-k}, & \text{otherwise, where } k = \min \{ j \in \omega : x_j \neq y_j \},
        \end{cases}
    \]
    for any $\boldsymbol{x} = (x_i), \boldsymbol{y} = (y_i) \in S^\omega$.
\end{de}

\begin{de}
    Let $S^\omega$ be the full shift space. The \textit{shift map} (or \textit{left shift}) $\sigma: S^\omega \to S^\omega$ is given by 
    \[
        \sigma(\boldsymbol{x}) = (x_{i+1})_{i \in \omega} \quad \text{for all } \boldsymbol{x} = (x_i)_{i \in \omega} \in S^\omega.
    \]
\end{de}

\begin{de}
    A pair $(X, \sigma)$ is called a \textit{subshift} (or simply \textit{shift}) if $X \subset S^\omega$ is a closed set and $\sigma(X) \subseteq X$, that is, $X$ is a closed, $\sigma$-invariant set in $S^\omega$.
\end{de}

Let $\boldsymbol{x} = (x_i) \in S^\omega$ and $n \in \mathbb{N}$. We denote by $x_n$ the $n$-th coordinate of $\boldsymbol{x}$, and by $\pi_n: S^\omega \to S$ the natural projection map onto the $n$-th coordinate.

Let $S$ be a finite alphabet and $(X, \sigma)$ a subshift of $(S^\omega, \sigma)$. We say that a word $u \in S^i$ is \textit{allowed} in $X$ if it appears as an initial block of some $\boldsymbol{x} \in X$, that is, $u = (x_0, \dots, x_{i-1})$. For each $i \in \omega$, we denote by $\mathcal{L}_i(X)$ the set of allowed words in $X$ of length $i$:
\[
    \mathcal{L}_i(X) = \left\{ (a_0, \dots, a_{i-1}) \in S^i : \text{ there is } \boldsymbol{x}\in X \text{ such that } (x_0, \dots, x_{i-1}) = (a_0, \dots, a_{i-1}) \right\}.
\]
We denote by $\mathcal{L}(X) = \bigcup_{i \in \omega} \mathcal{L}_i(X)$ the set of all allowed words in $X$.

It is known that a subshift can be defined in terms of \textit{forbidden} words. More precisely, given a set $F \subset \mathcal{L}(S^\omega)$, we define $X_F \subset S^\omega$ as the set of all sequences in $S^\omega$ that do not contain any word from $F$. It is easy to verify that $(X_F, \sigma)$ is a subshift of $S^\omega$. Conversely, given a subshift $(X, \sigma)$, if we set $F = \mathcal{L}(S^\omega) \setminus \mathcal{L}(X)$, then $X = X_F$.

\begin{de}\label{de4.6}
    A shift $X$ is said to be of \textit{$M$-step} for some $M \in \mathbb{N}$, if there exists a set of forbidden words $F$ such that $X = X_F$ and every word in $F$ has length $M+1$. Such a system is also called a \textit{subshift of finite type} (or simply shift of finite type). In particular, we say that a shift is \textit{$1$-step} if the length of the words in $F$ is $2$.
\end{de}

Note that although the following result can be derived using expansiveness and results from subsequent sections, our goal here is to provide a direct proof. 

\begin{thm}\label{thm4.8}
    Let $(X, f)$ be an expansive compact totally disconnected multivalued system. If $f$ has the shadowing property, then its orbit system is conjugate to a subshift of finite type. 
\end{thm}

\begin{proof}
Consider the following commutative diagram:
\begin{displaymath}
    \xymatrix{ 
        \mathrm{Orb}_f(X) \ar[rr]^{\sigma_f} \ar[d]_h & & \mathrm{Orb}_f(X) \ar[d]^h \\
        \Sigma \ar[rr]_\sigma & & \Sigma 
    }
\end{displaymath}
From the classical result, $(\mathrm{Orb}_f(X), \sigma_f)$ is conjugate to some subshift $\Sigma \subset S^\omega$ on a finite alphabet $S$. Hence, there is a conjugacy $h: \mathrm{Orb}_f(X) \to \Sigma$ such that $h \circ \sigma_f = \sigma \circ h$.

Let $d$ be a metric on $X$. We recall the compatible metric on $\mathrm{Orb}_f(X)$ defined by:
\[
    \rho((x_i), (y_i)) = \sum_{i \in \omega} \frac{d(x_i, y_i)}{2^{i+1}} \quad \text{for any} \ (x_i), (y_i) \in \mathrm{Orb}_f(X).
\]

Consider the uniform continuity of both $h$ and $h^{-1}$. Let $\varepsilon > 0$ be small enough, and let $\delta > 0$ be the constant witnessing the $\varepsilon$-shadowing property of $f$, and let $m \in \mathbb{N}$ be large enough such that:
\begin{align}\label{eqq4.1}
    \rho(\boldsymbol{x}, \boldsymbol{y}) < \varepsilon + \delta + \frac{1}{2^m} \implies d_\infty(h(\boldsymbol{x}), h(\boldsymbol{y})) < \frac{1}{4}
\end{align}
for all $\boldsymbol{x}, \boldsymbol{y} \in \mathrm{Orb}_f(X)$.

Similarly, let $n > m$ be large enough such that:
\begin{align}\label{eqq4.2}
    d_\infty(\boldsymbol{u}, \boldsymbol{v}) < \frac{1}{2^{n}} \implies \rho(h^{-1}(\boldsymbol{u}), h^{-1}(\boldsymbol{v})) < \frac{\delta}{2}
\end{align}
for all $\boldsymbol{u}, \boldsymbol{v} \in \Sigma$.

Let $W = \mathcal{L}_{2n}(S^\omega) \setminus \mathcal{L}_{2n}(h(\mathrm{Orb}_f(X)))$. From the definition of the subshift $X_W$, we have $h(\mathrm{Orb}_f(X)) \subset X_W$. Now, we claim that $X_W \subset h(\mathrm{Orb}_f(X))$.

Let $\boldsymbol{y}=(y_i) \in X_W$. Then for each $i \in \omega$, we have $(y_i, \dots, y_{i+2n-1}) \in \mathcal{L}_{2n}(h(\mathrm{Orb}_f(X)))$. Thus, there exists $\boldsymbol{x}_i \in h(\mathrm{Orb}_f(X))$ such that $(x_{i,0}, \dots, x_{i,2n-1}) = (y_i, \dots, y_{i+2n-1})$. This implies:
\begin{align}\label{eqmain}
    d_\infty(\sigma^j(\boldsymbol{x}_i), \sigma^{i+j}(\boldsymbol{y})) < \frac{1}{2^n} \quad \text{for each } 0 \leq j \leq n-1.
\end{align}
Furthermore, we have:
\[
    d_\infty(\sigma^j(\boldsymbol{x}_i), \boldsymbol{x}_{i+j}) < \frac{1}{2^n} \quad \text{for each } 0 \leq j \leq n-1.
\]
Then from \eqref{eqq4.2}, we get:
\[
    \rho(h^{-1}(\sigma^j(\boldsymbol{x}_i)), h^{-1}(\boldsymbol{x}_{i+j})) < \frac{\delta}{2} \quad \text{for each } 0 \leq j \leq n-1.
\]
Recall that $\Pi_0$ is the first coordinate projection from $\mathrm{Orb}_f(X)$ to $X$. Since $d(u_0, v_0) \leq 2\rho(\boldsymbol{u}, \boldsymbol{v})$, we obtain:
\[
    d(f(\Pi_0(h^{-1}(\boldsymbol{x}_i))), \Pi_0(h^{-1}(\boldsymbol{x}_{i+1}))) < \delta \quad \text{for all } i \in \omega.
\]
Also:
\[
    d(\Pi_0(h^{-1}(\sigma^j(\boldsymbol{x}_i))), \Pi_0(h^{-1}(\boldsymbol{x}_{i+j}))) < \delta \quad \text{for all } i \in \omega, \ 0 \leq j \leq n-1.
\]
As $f$ has the shadowing property, there is an $f$-orbit $\boldsymbol{z}=(z_i)$ such that:
\[
    d(z_i, \Pi_0(h^{-1}(\boldsymbol{x}_i))) < \varepsilon \quad \text{for all } i \in \omega.
\]
We estimate the distance:
\begin{align*}
   d(z_{i+j}, \Pi_0(\sigma_{f}^j(h^{-1}(\boldsymbol{x}_i)))) &= d(z_{i+j}, \Pi_0(h^{-1}(\sigma^j(\boldsymbol{x}_i)))) \\
   &< d(z_{i+j}, \Pi_0(h^{-1}(\boldsymbol{x}_{i+j}))) + d(\Pi_0(h^{-1}(\boldsymbol{x}_{i+j})), \Pi_0(h^{-1}(\sigma^j(\boldsymbol{x}_i)))) \\
   &< \varepsilon + \delta
\end{align*}
for all $i \in \omega$ and $0 \leq j \leq n-1$. Then we obtain:
\[
    \rho(\sigma_f^i(\boldsymbol{z}), h^{-1}(\boldsymbol{x}_i)) < \sum_{k=0}^n \frac{\varepsilon+\delta}{2^{k+1}} + \sum_{k=n+1}^\infty \frac{1}{2^k} < \varepsilon + \delta + \frac{1}{2^n} \quad \text{for all } i \in \omega.
\]
Hence, from \eqref{eqq4.1}, we have:
\[
    d_\infty(\sigma^i(h(\boldsymbol{z})), \boldsymbol{x}_i) = d_\infty(h(\sigma_f^i(\boldsymbol{z})), \boldsymbol{x}_i) < \frac{1}{4} \quad \text{for all } i \in \omega.
\]
Finally, combining this with \eqref{eqmain}, we have:
\begin{align*}
    d_\infty(\sigma^i(h(\boldsymbol{z})), \sigma^i(\boldsymbol{y})) &= d_\infty(h(\sigma_f^i(\boldsymbol{z})), \sigma^i(\boldsymbol{y})) \\
    &\leq d_\infty(\sigma^i(h(\boldsymbol{z})), \boldsymbol{x}_i) + d_\infty(\sigma^i(\boldsymbol{y}), \boldsymbol{x}_i) \\
    &< \frac{1}{4} + \frac{1}{2^n} < \frac{1}{2} \quad \text{for all } i \in \omega,
\end{align*}
which implies that $h(\boldsymbol{z}) = \boldsymbol{y}$. Therefore, $X_W = h(\mathrm{Orb}_f(X))$.
\end{proof}

\begin{rem}
    It can be shown that if $f \colon X \to 2^X$ is a continuous map satisfying the conditions in Theorem \ref{thm4.8}, then $f$ possesses the shadowing property if and only if it is conjugate to a subshift of finite type.
\end{rem}

\subsection{multivalued shadowing on compact Hausdorff spaces}
In \cite{GM}, Good and Meddaugh characterize the shadowing property for continuous single-valued maps using finite open covers. Their definition helps visualize the geometric nature of this property. To avoid complex computations, we adopt a similar approach for multivalued maps, defining the shadowing property in terms of open covers.

\begin{de}\label{def4.12}
    Let $(X, f)$ be a multivalued system and $\mathcal{U}$ be a finite open cover of $X$ (denoted $\mathcal{U} \in \mathcal{FOC}(X)$).
    \begin{enumerate}
        \item A sequence $(x_i)_{i \in \omega}$ is called a \textit{$\mathcal{U}$-pseudo-orbit} if for each $i \in \omega$, there exists $U_{i+1} \in \mathcal{U}$ such that:
        \[
            f(x_i) \cap U_{i+1} \neq \emptyset \quad \text{and} \quad x_{i+1} \in U_{i+1}.
        \]
       
        \item A sequence of open sets $(U_i)_{i \in \omega} \subseteq \mathcal{U}$ is called a \textit{$\mathcal{U}$-pseudo-orbit pattern} if there exists a sequence $(x_i)_{i \in \omega}$ in $X$ such that for all $i \in \omega$:
        \[
            x_i \in U_i \quad \text{and} \quad f(x_i) \cap U_{i+1} \neq \emptyset.
        \]
        
        \item A point $z \in X$ is said to \textit{$\mathcal{U}$-shadow} the sequence $(x_i)_{i \in \omega}$ if there exists an $f$-orbit $(z_i)_{i \in \omega}$ with $z_0 = z$ such that for each $i \in \omega$, there exists $U_i \in \mathcal{U}$ with:
        \[
            x_i, z_i \in U_i.
        \]
        (In this case, we also say that the orbit $(z_i)$ shadows $(x_i)$).
    \end{enumerate}
\end{de}

\begin{lem}\label{lem4.11}
    Let $X$ be a compact metric space. Then the multivalued map $f \colon X \to 2^X$ has the shadowing property if and only if for every finite open cover $\mathcal{U} \in \mathcal{FOC}(X)$, there exists an open cover $\mathcal{V} \in \mathcal{FOC}(X)$ such that every $\mathcal{V}$-pseudo-orbit is $\mathcal{U}$-shadowed by some $f$-orbit $(z_i)_{i\in\omega} \in \operatorname{Orb}_f(X)$.
\end{lem}

\begin{proof}
    Suppose $f$ has the shadowing property. Let $\mathcal{U} \in \mathcal{FOC}(X)$ and let $\varepsilon > 0$ be the Lebesgue number of $\mathcal{U}$. Then, for every $x \in X$, the ball $B_d(x, \varepsilon)$ is contained in some member of $\mathcal{U}$. Let $\delta > 0$ be the constant witnessing the $\varepsilon$-shadowing property for $f$.
    
    Now, let $\mathcal{V} \in \mathcal{FOC}(X)$ be a cover such that the diameter of each element in $\mathcal{V}$ is less than $\delta$. Let $\{x_i\}_{i\in\omega}$ be a $\mathcal{V}$-pseudo-orbit. By definition, for each $i$, there exists $V_{i+1} \in \mathcal{V}$ such that $f(x_i) \cap V_{i+1} \neq \emptyset$ and $x_{i+1} \in V_{i+1}$. Since the diameter of $V_{i+1}$ is less than $\delta$, we have $d(f(x_i), x_{i+1}) < \delta$. 
    Consequently, there exists an $f$-orbit $(z_i)_{i\in\omega} \in \operatorname{Orb}_f(X)$ such that $d(x_i, z_i) < \varepsilon$ for all $i \in \omega$. Since $\varepsilon$ is the Lebesgue number of $\mathcal{U}$, for each $i$, the set $\{x_i, z_i\}$ (whose diameter is less than $\varepsilon$) is contained in some set $U_i \in \mathcal{U}$. This shows that the sequence $\{x_i\}$ is $\mathcal{U}$-shadowed by $(z_i)$.

    Conversely, suppose that for each $\mathcal{U} \in \mathcal{FOC}(X)$, there exists $\mathcal{V} \in \mathcal{FOC}(X)$ such that every $\mathcal{V}$-pseudo-orbit is $\mathcal{U}$-shadowed by some $f$-orbit. Let $\varepsilon > 0$. Let $\mathcal{U} \in \mathcal{FOC}(X)$ be a cover consisting of open balls of radius $\varepsilon/2$. Let $\mathcal{V} \in \mathcal{FOC}(X)$ be the cover that witnesses the $\mathcal{U}$-shadowing. 
    Let $\delta > 0$ be the Lebesgue number of $\mathcal{V}$. Consider any $\delta$-pseudo-orbit $\{x_i\}_{i\in\omega}$. Since $d(f(x_i), x_{i+1}) < \delta$, there exists $y_i \in f(x_i)$ such that $d(y_i, x_{i+1}) < \delta$. There exists $V_{i+1} \in \mathcal{V}$ containing both $y_i$ and $x_{i+1}$. 
    By assumption, there exists an $f$-orbit $(z_i) \in \operatorname{Orb}_f(X)$ such that $x_i, z_i \in U_i$ for some $U_i \in \mathcal{U}$ for all $i$. Since the diameter of each $U_i$ is at most $\varepsilon$ (being a ball of radius $\varepsilon/2$), we have $d(x_i, z_i) < \varepsilon$ for all $i \in \omega$. Thus, $f$ has the shadowing property.
\end{proof}

\begin{thm}\label{thm3.12}
    Let $\mathcal{S} = \{g^\eta_\lambda, (X_\lambda, f_\lambda)\}_{\lambda \in \Lambda}$ be an inverse system of multivalued systems satisfying the conditions of Definition~\ref{def1.4}. Suppose that each system $(X_\lambda, f_\lambda)$ has the multivalued shadowing property and the induced inverse system of orbit spaces, $\{(g^\eta_\lambda)^*, (\operatorname{Orb}_{f_\lambda}(X_\lambda), \sigma_{f_\lambda})\}$, satisfies the Mittag-Leffler condition. Then the inverse limit system $(\varprojlim X_\lambda, \varprojlim f_\lambda)$ has the multivalued shadowing property.
\end{thm}

\begin{proof}
    Let $\mathcal{U}$ be a finite open cover of $X_\infty = \varprojlim \{ g^\eta_\lambda, X_\lambda \}$. By the definition of the topology of inverse limit, there exists $\lambda \in \Lambda$ and a finite open cover $\mathcal{W}_\lambda$ of $X_\lambda$ such that
    \[
        \mathcal{W} = \{ \pi_\lambda^{-1}(W) \cap X_\infty : W \in \mathcal{W}_\lambda \} \text{ is finer than } \mathcal{U}.
    \]
    
    Since the inverse system of orbit spaces $\{(g^\eta_\lambda)^*, (\operatorname{Orb}_{f_\lambda}(X_\lambda), \sigma_{f_\lambda})\}$ satisfies the Mittag-Leffler condition, there exists $\gamma \geq \lambda$ such that for all $\eta \geq \gamma$,
    \[
        (g^\gamma_\lambda)^*(\operatorname{Orb}_{f_\gamma}(X_\gamma)) = (g^\eta_\lambda)^*(\operatorname{Orb}_{f_\eta}(X_\eta)).
    \]

    Let $\mathcal{W}_\gamma = \{ (g^\gamma_\lambda)^{-1}(W) : W \in \mathcal{W}_\lambda \}$. Then $\mathcal{W}_\gamma$ is a finite open cover of $X_\gamma$. Since $(X_\gamma, f_\gamma)$ has the multivalued shadowing property, there exists a finite open cover $\mathcal{V}_\gamma$ of $X_\gamma$ such that every $\mathcal{V}_\gamma$-pseudo-orbit is $\mathcal{W}_\gamma$-shadowed by some $f_\gamma$-orbit.

    Define a cover for the limit space:
    \[
        \mathcal{V} = \{ \pi_\gamma^{-1}(V) \cap X_\infty : V \in \mathcal{V}_\gamma \}.
    \]
    Clearly, $\mathcal{V}$ is a finite open cover of $X_\infty$. Let $\{\boldsymbol{x}_i\}_{i \in \omega}$ be a $\mathcal{V}$-pseudo-orbit in $X_\infty$. This means for each $i$, there exists $V_{i+1} \in \mathcal{V}_\gamma$ such that
    \[
        \boldsymbol{x}_{i+1} \in \pi_\gamma^{-1}(V_{i+1}) 
    \]
    and there is an element $\boldsymbol{x}'\in \varprojlim f_\lambda(\boldsymbol{x}_i)$ such that $\boldsymbol{x}', \boldsymbol{x}_{i+1} \in \pi_\gamma^{-1}(V_{i+1})$.
    
    Projecting to $X_\gamma$, let $y_i = \pi_\gamma(\boldsymbol{x}_i)$. Then $\{y_i\}_{i \in \omega}$ forms a $\mathcal{V}_\gamma$-pseudo-orbit in $X_\gamma$ with pattern $\{V_i\}$. By the shadowing property of $f_\gamma$, there exists an orbit $\boldsymbol{z}^{(\gamma)} = \{z_i^{(\gamma)}\}_{i \in \omega} \in \operatorname{Orb}_{f_\gamma}(X_\gamma)$ such that $z_i^{(\gamma)}$ $\mathcal{W}_\gamma$-shadows $y_i$, i.e.,
    \[
        z_i^{(\gamma)} \in (g^\gamma_\lambda)^{-1}(W'_i) \quad \text{and} \quad y_i \in (g^\gamma_\lambda)^{-1}(W'_i)
    \]
    for some $W'_i \in \mathcal{W}_\lambda$ containing $\pi_\lambda(\boldsymbol{x}_i)$.
    
    Now, consider the projection of the pseudo-orbit and shadowing orbit to $X_\lambda$:
    \begin{equation}\label{eqqe4.4}
        z_i^{(\lambda)}=g^\gamma_\lambda(z_i^{(\gamma)}), \quad \pi_\lambda(\boldsymbol{x}_i)=g^\gamma_\lambda(\pi_\gamma(\boldsymbol{x}_i))\in g^\gamma_\lambda(W_i)=W_i'\in \mathcal{W}_\lambda.
    \end{equation}
    
    The sequence $\boldsymbol{z}^{(\lambda)} = \{z_i^{(\lambda)}\}$ is an element of $(g^\gamma_\lambda)^*(\operatorname{Orb}_{f_\gamma}(X_\gamma))$. By the Mittag-Leffler condition, there exists a limit orbit $\boldsymbol{Z} = \{\boldsymbol{Z}_\lambda\}_{\lambda \in \Lambda} \in \operatorname{Orb}_{\varprojlim f_\lambda}(X_\infty)$ such that
    \[
        \pi_\lambda(\boldsymbol{Z}_\lambda)_i = z_i^{(\lambda)} \quad \text{for all } i \in \omega.
    \]
    Notice that the sequence $\boldsymbol{Z}_i=\Pi_0\left(\sigma^i_{\varprojlim f_\lambda}(\{\boldsymbol{Z}_\lambda\}_{\lambda \in \Lambda})\right)$ forms a $\varprojlim f_\lambda$-orbit. Finally, from \eqref{eqqe4.4}, we obtain
   \[\boldsymbol{Z}_i,\boldsymbol{x}_i\in W\text{ for some }W\in \mathcal{W}.\]
   Hence, 
   \[\boldsymbol{Z}_i,\boldsymbol{x}_i\in U\text{ for some }U\in \mathcal{U}.\]
   Therefore, each $\mathcal{V}$-pseudo-orbit is $\mathcal{U}$-shadowed by some $\varprojlim f_\lambda$-orbit, implying $\varprojlim f_\lambda$ has multivalued shadowing property.
\end{proof}

\begin{rem}\label{rem4.14}
    The assumption of the Mittag-Leffler condition on inverse systems of orbit spaces is critical for Theorem \ref{thm3.12}.
\end{rem}

\section{Multivalued of finite type, graph cover and totally disconnected space}\label{se5}
\subsection{Graph cover}
Gambaudo and Martens~\cite{GM06} introduced graph covers for minimal Cantor systems, it is extended to all Cantor systems by Shimomura~\cite{Sh14,Sh16,Sh20}. Here, we briefly review the relevant definitions and properties of Shimomura's.

Let $V$ be a finite set endowed with the discrete topology and $E \subset V \times V$ a relation on $V$. We call the pair $G=(V,E)$ a finite (directed) \textit{graph}, where $V$ is the set of \textit{vertices} and $E$ is the set of \textit{edges}. Throughout this paper, we assume that for each $u \in V$, there exists a vertex $v \in V$ such that $(u,v) \in E$.

\begin{de}\label{def5.1}
    Let $G_1=(V_1,E_1)$ and $G_2=(V_2,E_2)$ be two graphs. A map $\varphi \colon V_1 \to V_2$ is said to be a \textit{graph homomorphism} if 
    \[
        (u,v) \in E_1 \implies (\varphi(u),\varphi(v)) \in E_2.
    \]
    We say that $\varphi$ is \textit{edge-surjective} if $(\varphi \times \varphi)(E_1) = E_2$. Furthermore, $\varphi$ is said to be \textit{$+$directional} if 
    \[
        (u,v), (u,v') \in E_1 \implies \varphi(v) = \varphi(v').
    \]
\end{de}

\begin{rem}
    Moreover, we say that a graph homomorphism is \textit{pointwise edge-surjective} if for any $v \in V_1$, the map $\varphi$ sends the set of outgoing edges from $v$ surjectively onto the set of outgoing edges from $\varphi(v)$. It is easy to verify that edge-surjectivity is not equivalent to pointwise edge-surjectivity.
\end{rem}

\begin{de}
    A graph homomorphism $\varphi \colon G_1 \to G_2$ is called a \textit{cover} if it is a $+$directional and edge-surjective graph homomorphism.
\end{de}

\begin{de}
    Let $G_0 \xleftarrow{\varphi_0} G_1 \xleftarrow{\varphi_1} G_2 \xleftarrow{\varphi_2} \cdots$ be a sequence of $+$directional edge-surjective graph homomorphisms, where $G_i=(V_i,E_i)$ for all $i \in \mathbb{N}_0$. Define the inverse limit vertex space by
    \[
        V_\mathcal{G} = \left\{ (x_i)_{i\ge 0} \in \prod_{i=0}^\infty V_i : x_i = \varphi_i(x_{i+1}) \text{ for all } i \ge 0 \right\}
    \]
    and the edge relation by
    \[
        E_\mathcal{G} = \left\{ (\boldsymbol{x}, \boldsymbol{y}) \in V_\mathcal{G} \times V_\mathcal{G} : (x_i, y_i) \in E_i \text{ for all } i \ge 0 \right\}.
    \]
    The dynamical action is defined by the relation
    \[
        \boldsymbol{x} \mapsto E_\mathcal{G}(\boldsymbol{x}) = \{ \boldsymbol{y} \in V_\mathcal{G} : (\boldsymbol{x}, \boldsymbol{y}) \in E_\mathcal{G} \}.
    \]
    We call the pair $(V_\mathcal{G}, E_\mathcal{G})$ a \textit{graph cover} (or an inverse limit of graph covers).
\end{de}

We have the following result (cf. \cite{Sh14,Sh16,Sh20}).

\begin{thm}
    A single-valued metrizable topological dynamical system $(X,f)$ is zero-dimensional if and only if it is topologically conjugate to a graph cover $(V_\mathcal{G}, E_\mathcal{G})$.
\end{thm}

\subsection{Multivalued maps on compact totally disconnected Hausdorff spaces}

In this subsection, we consider multivalued maps on a compact totally disconnected Hausdorff space. We represent these systems as an inverse limit of an inverse system composed of elementary multivalued maps. Notably, we find that this approach coincides with the construction of a graph cover. Moreover, by invoking the \textit{Closed Graph Theorem} (Proposition \ref{prop2.1}), we demonstrate that constructing a graph cover is a method for approximating the graph of a multivalued map by finite directed graphs.

\begin{de}
    Let $V$ be a finite set endowed with the discrete topology. A multivalued system $(V,f)$, where $f \colon V \to 2^V$, is called a \textit{multivalued system of finite type}.
\end{de}

\begin{rem}
    Let $f \colon V \to 2^V$. We define a directed graph $G=(V, E_V)$ by the condition:
    \[ (u,v) \in E_V \iff v \in f(u). \]
    Conversely, if $G=(V, E_V)$ is a finite directed graph, we define the map $f \colon V \to 2^V$ by 
    \[ f(u) = \left\{ v \in V  \; \middle |\; (u,v) \in E_V \right\}. \]
\end{rem}

Given a multivalued system of finite type $(V,f)$, a sequence $\boldsymbol{x}=(x_0, x_1, \dots, x_i, \dots)$ with $x_i \in V$ belongs to the inverse limit space ${\rm{Orb}}_f(V)$ if and only if $x_i \in f(x_{i-1})$ for all $i \in \omega$. This implies that the system $(\operatorname{Orb}_f(V), \sigma_f)$ is conjugate to a 1-step shift. Consequently, we have the following result:

\begin{lem}\label{lem4.7}
    Let $f \colon V \to 2^V$ be a set-valued map of finite type. Then:
    \begin{enumerate}
        \item $f$ possesses the multivalued shadowing property;
        \item its orbit system $(\operatorname{Orb}_f(V), \sigma_f)$ is a subshift of finite type.
    \end{enumerate}
\end{lem}

\begin{proof}
    The result follows directly from the definitions.
\end{proof}

Let $(\Sigma, \sigma)$ be a subshift of finite type of $M$-step. Let $X_M$ denote the set of allowed words of length $M+1$ in $\Sigma$. As presented in \cite{PW1}, we can define a conjugacy via a \textit{sliding block code} $\phi \colon \Sigma \to (X_M)^{\mathbb{Z}}$ as follows:
\[
\phi((x_i)_{i \in \mathbb{Z}}) = 
\left( 
\begin{pmatrix}
x_0 \\
x_1 \\
\vdots \\
x_{M}
\end{pmatrix},
\begin{pmatrix}
x_1 \\
x_{2} \\
\vdots \\
x_{M+1}
\end{pmatrix},
\begin{pmatrix}
x_{2} \\
x_{3} \\
\vdots \\
x_{M+2}
\end{pmatrix},
\dots \right).
\]
We define a multivalued map $f \colon X_M \to 2^{X_M}$ by the condition: $v \in f(u)$ if and only if 
\[ u = y_1 y_2 \dots y_{M+1} \quad \text{and} \quad v = y_2 y_3 \dots y_{M+2}, \]
where $y_1 \dots y_{M+2}$ is a valid word in $\Sigma$. Since $(\Sigma, \sigma)$ is a subshift of finite type, it is conjugate to the orbit system $(\operatorname{Orb}_f(X_M), \sigma_f)$. Consequently, we have the following result:

\begin{lem}
    Every subshift of finite type is conjugate to an orbit system of a multivalued system of finite type.
\end{lem}

Let $X$ be a compact totally disconnected Hausdorff space. Recall that $\mathcal{P}art(X)$ denotes the collection of all finite clopen partitions of $X$, which forms a cofinal subset of $\mathcal{FOC}(X)$. We regard each $\mathcal{U} \in \mathcal{P}art(X)$ as a finite alphabet, and endow it with the discrete topology.

If $\mathcal{V} \in \mathcal{P}art(X)$ is a refinement of $\mathcal{U}$ (denoted $\mathcal{U} \preceq \mathcal{V}$), we define the canonical projection $\pi_{\mathcal{U}}^{\mathcal{V}} \colon \mathcal{V} \to \mathcal{U}$ by letting $\pi_{\mathcal{U}}^{\mathcal{V}}(V)$ be the unique element $U \in \mathcal{U}$ such that $V \subseteq U$. Since the bonding maps $\pi_{\mathcal{U}}^{\mathcal{V}}$ are surjective, we have the following observation:

\begin{lem}\label{lem4.9}
    The inverse system $\{(\mathcal{U}, \pi_{\mathcal{U}}^{\mathcal{V}} )_{\mathcal{U} \in \mathcal{P}(X)}\}$ satisfies the Mittag-Leffler condition.
\end{lem}

Now, let $(X, f)$ be a multivalued dynamical system. For each $\mathcal{U} \in \mathcal{P}(X)$, we define the induced multivalued map $f_\mathcal{U} \colon \mathcal{U} \to 2^\mathcal{U}$ by
\[
    f_\mathcal{U}(U) = \left\{V \in \mathcal{U} \;\middle|\; f(U) \cap V \neq \emptyset\right\}.
\]

\begin{lem}\label{lem5.10}
    If $\mathcal{U} \preceq \mathcal{V}$, then the map $\pi^{\mathcal{V}}_{\mathcal{U}} \colon (\mathcal{V}, f_\mathcal{V}) \to (\mathcal{U}, f_\mathcal{U})$ is a subfactor map in the sense of Definition~\ref{def2.10}. Consequently, $\pi^{\mathcal{V}}_{\mathcal{U}}$ satisfies Condition~\eqref{eq1.11}.
\end{lem}

Recall that the graph $G_f$ of a multivalued map $f \colon X \to 2^X$ is defined by 
\[
    G_f= \{(x,y) \in X \times X \mid y \in f(x)\}.
\]
Let $(X,f)$ be a multivalued system on a totally disconnected space and let $\{\pi^\mathcal{V}_\mathcal{U}, (\mathcal{U}, f_{\mathcal{U}})\}$ be the associated inverse system. We define the inverse limit space
\[
    X_{\mathcal{P}art(X)} = \left\{ (U_{\mathcal{U}}) \in \prod_{\mathcal{U} \in \mathcal{P}art(X)} \mathcal{U} \;\middle|\; \pi^\mathcal{V}_{\mathcal{U}}(U_\mathcal{V}) = U_\mathcal{U} \text{ whenever } \mathcal{U} \preceq \mathcal{V} \right\}
\]
and the limit graph
\[
    G_{f_{\mathcal{P}art(X)}} = \left\{ (\mathbf{u}, \mathbf{v}) \in X_{\mathcal{P}art(X)} \times X_{\mathcal{P}art(X)} \;\middle|\; (u_\mathcal{U}, v_\mathcal{U}) \in G_{f_\mathcal{U}} \text{ for all } \mathcal{U} \in \mathcal{P}art(X) \right\},
\]
where $\mathbf{u}=(u_\mathcal{U})$ and $\mathbf{v}=(v_\mathcal{U})$. Here, each $\mathcal{U}$ is endowed with the discrete topology, and $X_{\mathcal{P}art(X)}$ inherits the product topology.

\begin{thm}\label{thm4.11}
    Let $(X,f)$ be a multivalued system on a compact totally disconnected Hausdorff space $X$. Then the following hold:
    \begin{enumerate}
        \item $(X_{\mathcal{P}art(X)}, f_{\mathcal{P}art(X)})$ is an upper semicontinuous multivalued system;
        \item $(X,f)$ is topologically conjugate to $(X_{\mathcal{P}art(X)}, f_{\mathcal{P}art(X)})$.
    \end{enumerate}
\end{thm}

\begin{proof}
    (1) This follows directly from Lemmas \ref{lem4.9} and \ref{lem5.10}, and Theorem \ref{lem1.12}.
    
    (2) We define a map $g \colon X_{\mathcal{P}art(X)} \to X$ by
    \[
        (U_\mathcal{U})_{\mathcal{U}} \mapsto \bigcap_{\mathcal{U} \in \mathcal{P}(X)} U_\mathcal{U}.
    \]
    Since $X$ is a compact Hausdorff space and the elements of the partitions are clopen, it is straightforward to verify that $g$ is a continuous bijection. Since $X_{\mathcal{P}art(X)}$ is compact and $X$ is Hausdorff, $g$ is a homeomorphism.

    Let $x,y \in X$. Let $\mathbf{x} = (U_\mathcal{U}(x))$ and $\mathbf{y} = (U_\mathcal{U}(y))$ be the unique coherent sequences in $X_{\mathcal{P}art(X)}$ such that $g(\mathbf{x})=x$ and $g(\mathbf{y})=y$.
    
    First, suppose $\mathbf{y} \in f_{\mathcal{P}art(X)}(\mathbf{x})$. By definition, this implies $f(U_\mathcal{U}(x)) \cap U_\mathcal{U}(y) \neq \emptyset$ for all $\mathcal{U} \in \mathcal{P}art(X)$. Assume, for the sake of contradiction, that $y \notin f(x)$. Since $X$ is Hausdorff and $f(x)$ is closed (by upper semicontinuity), there exist disjoint open neighborhoods $V$ of $y$ and $W$ of $f(x)$, repsctively. Since $f$ is upper semicontinuous at $x$, there exists an open neighborhood $O$ of $x$ such that $f(O) \subset W$. 
    
    Since the family of clopen partitions is cofinal (forms a base for the topology), we can choose a partition $\mathcal{K} \in \mathcal{P}art(X)$ sufficiently fine such that $U_\mathcal{K}(y) \subset V$ and $U_\mathcal{K}(x) \subset O$. Consequently,
    \[
        f(U_\mathcal{K}(x)) \subset f(O) \subset W.
    \]
    Since $U_\mathcal{K}(y) \subset V$ and $V \cap W = \emptyset$, it follows that $f(U_\mathcal{K}(x)) \cap U_\mathcal{K}(y) = \emptyset$, which contradicts the assumption that $\mathbf{y} \in f_{\mathcal{P}art(X)}(\mathbf{x})$. Thus, $y\in f(x)$

    Conversely, suppose $y \in f(x)$. By Lemma \ref{lem4.9}, for every $\mathcal{U} \in \mathcal{P}art(X)$, since $y \in f(x)$ and $x \in U_\mathcal{U}(x)$, we have $f(U_\mathcal{U}(x)) \cap U_\mathcal{U}(y) \neq \emptyset$. This implies
    \[
        \mathbf{y} \in f_{\mathcal{P}art(X)}(\mathbf{x}).
    \]
    Therefore, $g$ is a multivalued conjugacy between $(X,f)$ and $(X_{\mathcal{P}art(X)}, f_{\mathcal{P}art(X)})$.
\end{proof}

\begin{cor}
    Let $X$ be a compact totally disconnected Hausdorff space and $R$ is a closed relation on $X$. Then $R$ is conjugate to an inverse limit of homomorphisms of finite directed graphs.
\end{cor}

\begin{rem}
    By Lemma \ref{lem4.7} and Theorem \ref{thm4.11}, the assumption that the inverse system of orbit systems satisfies the Mittag-Leffler condition in Theorem \ref{thm3.12} is sharp. 
\end{rem}

Conversely, we consider how multivalued systems of finite type (or graphs) generate a general compact totally disconnected Hausdorff multivalued system.

Let $\varphi$ be a graph homomorphism from $G_1=(V_1,E_1)$ to $G_2=(V_2,E_2)$. Let $f_1 \colon V_1 \to 2^{V_1}$ and $f_2 \colon V_2 \to 2^{V_2}$ be the induced multivalued systems of finite type. We define a map $g \colon V_1 \to V_2$ by $g(v) = \varphi(v)$.
Since $(u,v) \in E_1$ implies $(\varphi(u),\varphi(v)) \in E_2$, we have 
\[
    g(f_1(u)) \subseteq f_2(g(u)),
\]
which implies that $g$ satisfies Condition \eqref{eq1.11}. Hence, combining Theorem \ref{lem1.12} and Theorem \ref{thm4.11}, we obtain the following result. This result also refines Shimomura's work on single-valued zero-dimensional systems.

\begin{thm}
    $(X,f)$ is a compact totally disconnected Hausdorff multivalued system if and only if it is the inverse limit of an inverse system of multivalued systems of finite type (or graphs) satisfying the Mittag-Leffler condition. In particular, if the inverse system is countable and $+$directional, then the inverse limit is conjugate a single-valued continuous system.
\end{thm}

\section{Characterizing multivalued shadowing on totally disconnected space}\label{se6}
\subsection{Shifts generated by finite open covers}\label{subse6.1} 
For convenience, we borrow some notation from Section \ref{se:3.1} for further discussion.
Let $(X,f)$ be a multivalued system and $\mathcal{U}\in \mathcal{FOC}(X)$ and $\mathbf{U}=\{U_i\}$ a sequence of elements in $\mathcal{U}$. Recall the notion of orbit-discriminant defined in Section \ref{se:3.1}.
  \[
            [\mathbf{U}] := \bigcap_{n=0}^{\infty} [U_0, \dots, U_n] = U_0 \cap f^{-1}\left( U_1 \cap f^{-1}\left( U_2 \cap \cdots \right) \right),
        \]
if $[\mathbf{U}]$ is not empty, for each $x\in [\mathbf{U}]$ there is an $f$-orbit $(x_i)\in {\rm{Orb}}_f(X)$ with $x_0=x$ passing through $\{U_i\}$.
Let $\mathcal{U}\in \mathcal{FOC}(X)$. Define
\[
\mathcal{U}^\omega = \{ (U_i) : U_i \in \mathcal{U} \}.
\]
The space $\mathcal{U}^\omega$ inherits the product topology and forms a one-sided shift space over the alphabet $\mathcal{U}$ with the shift map \[\sigma \colon (U_i) \mapsto (U_{i+1}).\]

We denote by $\mathcal{O}(\mathcal{U}) \subset \mathcal{U}^\omega$ the closure of the set of all sequences $\mathbf{U}=(U_i)$ in $\mathcal{U}$ such that $[\mathbf{U}] \neq \emptyset$. Additionally, we define
\[
    \mathcal{PO}(\mathcal{U}) = \left\{ (U_i) \in \mathcal{U}^\omega \;\middle|\; f(U_i) \cap U_{i+1} \neq \emptyset \text{ for all } i \in \omega \right\},
\]
and
\[
    \mathcal{OC}(\mathcal{U}) = \left\{ (U_i) \in \mathcal{U}^\omega \;\middle|\; \text{there exists } x \in X \text{ such that } x \in [\overline{U_i}] \right\}.
\]

\begin{lem}\label{lem5.1}
    Let $(X, f)$ be a multivalued system and $\mathcal{U} \in \mathcal{FOC}(X)$. Then the following hold:
    \begin{enumerate}
        \item $\mathcal{O}(\mathcal{U}) \subseteq \mathcal{PO}(\mathcal{U})$;
        \item $\mathcal{O}(\mathcal{U}) = \overline{\{ (U_i) \in \mathcal{U}^\omega \mid [U_i] \neq \emptyset \}} = \bigcap_{n \in \omega} \{ (U_i) \in \mathcal{U}^\omega \mid [U_0, \dots, U_n] \neq \emptyset \}$;
        \item $(\mathcal{O}(\mathcal{U}), \sigma)$ is a subshift of $(\mathcal{U}^\omega, \sigma)$;
        \item $(\mathcal{PO}(\mathcal{U}), \sigma)$ is an $1$-step subshift of $(\mathcal{U}^\omega, \sigma)$;
        \item $(\mathcal{OC}(\mathcal{U}), \sigma)$ is a subshift of $(\mathcal{U}^\omega, \sigma)$ containing $\mathcal{O}(\mathcal{U})$.
    \end{enumerate}
\end{lem}
\begin{proof}
    Statements 1.-- 4. follow directly from the definitions.

    5. Let $\{ (U_i^{(j)}) \}_{j\in\omega}$ be a sequence in $\mathcal{OC}(\mathcal{U})$ converging to a point $(W_i) \in \mathcal{U}^\omega$. For each $j\in\omega$, there exists an orbit $(z_i^{(j)}) \in \operatorname{Orb}_f(X)$ such that $z_i^{(j)} \in \overline{U_i^{(j)}}$ for each $i\in\omega$. By closed graph theorem, passing to a subsequence if necessary, we may assume that $\{ (z_i^{(j)}) \}$ converges to some point $(z_i') \in \operatorname{Orb}_f(X)$. 
    
    Fix an index $i \in \omega$. Since $\{ (U_i^{(j)}) \}$ converges to $(W_i)$ in the product topology (where $\mathcal{U}$ is discrete), we have $U_i^{(j)} = W_i$ for all sufficiently large $j$. Consequently, $z_i^{(j)} \in \overline{W_i}$ for such $j$. Since $\overline{W_i}$ is closed and $z_i^{(j)}$ converges to $z_i'$, it follows that $z_i' \in \overline{W_i}$. 
    
    This implies that $(\overline{W_i})$ tracks the orbit $(z_i')$, and thus $(W_i) \in \mathcal{OC}(\mathcal{U})$. Therefore, $\mathcal{OC}(\mathcal{U})$ is closed.
    
    By a similar argument, it is straightforward to verify that $\mathcal{O}(\mathcal{U}) \subseteq \mathcal{OC}(\mathcal{U})$.
\end{proof}

Let $X$ be a compact Hausdorff space (not necessarily metrizable). For any $\mathcal{U}, \mathcal{V} \in \mathcal{FOC}(X)$, recall that $\mathcal{V} \preceq \mathcal{U}$ if $\mathcal{U}$ refines $\mathcal{V}$. Then $(\mathcal{FOC}(X), \preceq)$ is a directed set. By the Axiom of Choice, we have the following statement.
\begin{thm}\label{thm5.2}
    Let $f \colon X \to 2^X$ be a multivalued map on a compact Hausdorff space $X$, let $\{\mathcal{U}_\lambda\}_{\lambda \in \Lambda}$ be a cofinal directed subset of $\mathcal{FOC}(X)$. Then the following hold:
    \begin{enumerate}
        \item For each $x \in X$, there exists a family $\{U_\lambda(x)\}_{\lambda \in \Lambda}$ with $U_\lambda(x) \in \mathcal{U}_\lambda$ such that $\{x\} = \bigcap_{\lambda \in \Lambda} U_\lambda(x)$. Moreover, for any such family, we have
        \[
            f^n(x) = \bigcap_{\lambda \in \Lambda} \pi_0 \left( \sigma^n \left( \mathcal{O}(\mathcal{U}_\lambda) \cap \pi_0^{-1}(U_\lambda(x)) \right) \right)
        \]
        for all $n \in \omega$.
        \item For any $f$-orbit $(x_i)$, there exists a family of sequences $\{(U_i^\lambda)\}_{\lambda \in \Lambda}$ with $U_i^\lambda \in \mathcal{U}_\lambda$ such that $(U_i^\lambda)$ is an orbit pattern of $(x_i)$ for each $\lambda \in \Lambda$, and we have
        \[
            x_n = \bigcap_{\lambda \in \Lambda} \pi_0 \left( \sigma^n \left( (U_i^\lambda) \right) \right).
        \]
    \end{enumerate}
\end{thm}
\begin{proof}
  1. Let $\Lambda=\mathcal{FOC}(X)$ and $\{\mathcal{U}_\lambda:\mathcal{U}_\lambda\in \mathcal{FOC}(X)\}$ be given. For each $x\in X$ and $\lambda\in \Lambda$, choose $U_\lambda(x)\in \mathcal{U}_\lambda$ with $x\in {U}_\lambda(x)$. Then $x\in \bigcap U_\lambda(x)$ and $\bigcap U_\lambda(x)$ is a singleton.

   On the one hand,  given $x_n\in f^n(U_\lambda(x))$, there is a finite $f$-orbit $(x_0,x_1,\cdots,x_n)$ with $x_0\in U_\lambda(x)$, suppose $x_i\in U_i^\lambda$ for $1\leq j\leq n$, then $(x_0,x_1,\cdots,x_n)$ has $(U_\lambda(x),U_1^\lambda,\cdots,U_n^\lambda)$-orbit pattern, and extend the orbit by filling points such that $x_{j+1}\in f(x_j)$ for all $j\geq n$, Suppose $x_j\in U^\lambda_j$ for each $j\geq n+1$, then we obtain an $f$-orbit $(x_0,x_1,\cdots,x_n,\cdots)$ with $(U_\lambda(x),U_1^\lambda,\cdots,U_n^\lambda,\cdots)$-orbit pattern.  Hence, 
   \begin{align}\label{eqqq6.1}
       x_n\in \pi_0\left(\sigma^n\left(\mathcal{O}(\mathcal{U}_\lambda)\cap \pi^{-1}_0(U_\lambda(x))\right)\right).
   \end{align}

    Now, suppose that $y\in \bigcap\pi_0\left(\sigma^n\left(\mathcal{O}(\mathcal{U}_\lambda)\cap \pi^{-1}_0(U_\lambda(x))\right)\right)$, but $y\notin f^n(x)$. As $f$ is upper semicontinuous (so is $f^n$), there is $\lambda\in\Lambda$ such that
    \[U_\lambda(y)\cap f^n(U_\lambda(x))=\emptyset.\]
    But $U_\lambda(y)\cap \pi_0\left(\sigma^n\left(\mathcal{O}(\mathcal{U}_\lambda)\cap \pi^{-1}_0(U_\lambda(x))\right)\right)\ne\emptyset$, hence, there is an $(U_i^\lambda)$-orbit pattern with $U_n^\lambda=U_\lambda(y)$, which means that 
    \[U_\lambda(y)\bigcap \left(f^n([U_\lambda(x),U_1^\lambda,\cdots,U_n^\lambda])\subset f^n(U_\lambda(x))\right)\ne\emptyset,\]
    which is a contradiction, so $y\in f^n(x)$. Thus, \[\bigcap\pi_0\left(\sigma^n\left(\mathcal{O}(\mathcal{U}_\lambda)\cap \pi^{-1}_0(U_\lambda(x))\right)\right)\subset f^n(x).\] So, we have 
    \[f^n(x)=f^n\left(\bigcap U_\lambda(x)\right)\subset\bigcap f^n\left(U_\lambda(x)\right)\overset{\eqref{eqqq6.1}}{\subset} \bigcap\pi_0\left(\sigma^n\left(\mathcal{O}(\mathcal{U}_\lambda)\cap \pi^{-1}_0(U_\lambda(x))\right)\right).\]
   Hence, 
   \[f^n(x)=\bigcap\pi_0\left(\sigma^n\left(\mathcal{O}(\mathcal{U}_\lambda)\cap \pi^{-1}_0(U_\lambda(x))\right)\right).\]

    2. It is easy to see that \[x_n=\Pi_0(\sigma_f^n((x_i)))\in \bigcap\pi_0\left(\sigma^n\left((U^\lambda_i)\right)\right).\]
   Conversely, suppose $z\in \bigcap\pi_0\left(\sigma^n\left((U^\lambda_i)\right)\right)$. Then 
    \[z\in \bigcap U_n^\lambda,\]
    but $\bigcap U_n^\lambda$ contains only one point. Thus, $z=x_n$.
\end{proof}

\subsection{Multivalued shadowing property on totally disconnected spaces}
In this subsection, we consider the multivalued shadowing property on a compact totally disconnected Hausdorff space $X$ (not necessarily metrizable or expansive).

Since each $\mathcal{U} \in \mathcal{P}art(X)$ consists of clopen subsets of $X$, we have
\[
    \mathcal{O}(\mathcal{U}) = \{ (U_i) \in \mathcal{U}^\omega : [U_i] \neq \emptyset \}.
\]
Let $\mathcal{U}, \mathcal{V} \in \mathcal{P}art(X)$ with $\mathcal{U} \preceq \mathcal{V}$. The map $\pi_\mathcal{U}^\mathcal{V} \colon \mathcal{V} \to \mathcal{U}$ naturally induces a continuous map \[\pi_\mathcal{U}^\mathcal{V} \colon \mathcal{V}^\omega \to \mathcal{U}^\omega\] defined by \[\pi^\mathcal{V}_\mathcal{U}((V_i)) = (\pi^\mathcal{V}_\mathcal{U}(V_i)) = (U_i)\quad (\text{where } V_i \subset U_i).\] Thus, from the tracking orbit method (see Figure \ref{F1}), we obtain,
\begin{lem}\label{lem5.3}{\rm{(cf. \cite{GM})}}
    Let $f:X\to 2^X$ be a multivalued map on compact totally disconnected Hausdorff space $X$ and let $\mathcal{U},\mathcal{V}\in \mathcal{P}art(X)$ with $\mathcal{U}\preceq \mathcal{V}$. Then the shift action $\sigma$ commutes with $\pi^\mathcal{V}_\mathcal{U}$, and 
    \begin{enumerate}
        \item $\pi^\mathcal{V}_\mathcal{U}(\mathcal{O}(\mathcal{V}))=\mathcal{O}(\mathcal{U})$;
        \item $\mathcal{O}(\mathcal{U})\subset \pi^\mathcal{V}_\mathcal{U}(\mathcal{PO}(\mathcal{V}))\subset \mathcal{PO}(\mathcal{U})$.
        \item $\mathcal{OC}(\mathcal{U})=\mathcal{O}(\mathcal{U}).$
    \end{enumerate}
\end{lem}
\begin{proof}
   It follows from the definition that $\pi^\mathcal{V}_\mathcal{U}$ and $\sigma$ commutes. 

    1. Consider $(V_i)\in \mathcal{O}(\mathcal{V})$. There is an $f$-orbit $(x_i)$ such that $(x_i)$ is $(V_i)$-orbit pattern. Then $x_i\in \pi^\mathcal{V}_\mathcal{U}(V_i)=U_i$, so that $(U_i)=(\pi^\mathcal{V}_\mathcal{U}(V_i))\in \mathcal{O}(\mathcal{U})$.

    Conversely, suppose $(x_i)$ has $(U_i)$-orbit pattern for some $(U_i)\in \mathcal{O}(\mathcal{U})$, and $x_i\in V_i$ for some $V_i\in \mathcal{V}$. Then as $\mathcal{V}$ and $\mathcal{U}$ are partition, we have $(\pi^\mathcal{V}_\mathcal{U}(V_i))=(U_i)$. Thus, 1. holds.

    Statements 2. and 3. follow from Lemma \ref{lem5.1}.
\end{proof}

\begin{thm}\label{thm6.4}
    Let $f:X\to 2^X$ be a multivalued map on compact Hausdorff totally disconnected space $X$. Let $\{\mathcal{U}_\lambda\}_{\lambda\in\Lambda}$ be a cofinal directed family of $\mathcal{P}art(X)$. Then 
    \begin{enumerate}
        \item the system $(X,f)$ is conjugate to $(X_{\mathcal{P}art(X)}, f_{\mathcal{P}art(X)})$ defined as in Section \ref{se5}.
        \item Its orbit system $({\rm{Orb}}_f(X),\sigma_f)$ is conjugate to $(\underset{\longleftarrow}{\lim}\{\mathcal{O}(\mathcal{U}_\lambda),\pi^\mathcal{V}_\mathcal{U}\},\underset{\longleftarrow}{\lim}\sigma_\lambda)$ by the map
\begin{align}\label{eq5.1}
       (\boldsymbol{U}_\lambda)_{\lambda\in \Lambda}\mapsto \left(\bigcap \pi_0\left(\sigma^i\left(\boldsymbol{U}_\lambda\right)\right)\right)_{i\in\omega}.
\end{align}
     for all $(\boldsymbol{U}_\lambda)_{\lambda\in \Lambda}\in \underset{\longleftarrow}{\lim}\{\mathcal{O}(\mathcal{U}_\lambda),\pi^\mathcal{V}_\mathcal{U}\}$.
    \end{enumerate}
\end{thm}
\begin{proof}
    1. This follows directly from Theorem \ref{thm4.11}.

    2. Suppose $(x_i)$ is an $f$-orbit. Since $\mathcal{U}_\lambda$ is a clopen partition of $X$ for each $\lambda\in \Lambda$, thus, thee is only one choice $(U_i^\lambda)$ such that $x_i\in U_i^\lambda$, which means that $(x_i)$ has only one choice of $(U_i^\lambda)$-orbit pattern.  Hence, by Theorem \ref{thm5.2}, the point $\left(\bigcap \pi_0\left(\sigma^i\left(\boldsymbol{U}_\lambda\right)\right)\right)_{i\in \omega}$ is an $f$-orbit and the map \eqref{eq5.1} is a well-defined injective onto map. The continuity follows directly from Theorem \ref{thm5.2}, too. Therefore, \eqref{eq5.1} is a conjugacy.
\end{proof}

The following result shows that if the orbital space of Gambaudo and Martens‘ graph cover (refer to Remark \ref{rem3.3}) virtually coincides with the orbital space of Shimomura's graph cover, then the graph cover itself has the shadowing property (cf. \cite[Lemma 15]{GM}). 

\begin{lem}\label{lem5.5}
    Let $f \colon X \to 2^X$ be a multivalued map on a compact totally disconnected Hausdorff space $X$. Then $f$ has the multivalued shadowing property if and only if for each $\mathcal{U} \in \mathcal{P}art(X)$, there exists $\mathcal{V} \in \mathcal{P}art(X)$ finer than $\mathcal{U}$ such that for all $\mathcal{W} \in \mathcal{P}art(X)$ finer than $\mathcal{V}$, we have 
    \[
        \pi^\mathcal{W}_\mathcal{U}(\mathcal{PO}(\mathcal{W})) = \mathcal{O}(\mathcal{U}).
    \]
\end{lem}

\begin{proof}
    Suppose $f$ has the multivalued shadowing property. Let $\mathcal{U} \in \mathcal{P}art(X)$. Choose $\mathcal{V} \in \mathcal{P}art(X)$ with $\mathcal{U} \preceq \mathcal{V}$ such that any sequence $\{x_i\}$ with a $\mathcal{V}$-pseudo-orbit pattern is shadowed by some $f$-orbit with a $\mathcal{U}$-orbit pattern.

    Consider a $\mathcal{V}$-pseudo-orbit $\{x_i\}$ with pseudo-orbit pattern $(V_i)$. By the shadowing assumption, there exists $(z_i) \in \operatorname{Orb}_f(X)$ such that $x_i, z_i \in U_i$ for some $U_i \in \mathcal{U}$ for all $i \in \omega$. Thus, $V_i \cap U_i \neq \emptyset$. Since $\mathcal{V}$ is finer than $\mathcal{U}$ and the elements of $\mathcal{U}$ are disjoint (as it is a partition), we obtain $V_i \subseteq U_i$. Hence, $\pi^\mathcal{V}_\mathcal{U}((V_i)) = (U_i)$. This implies $\pi^\mathcal{V}_\mathcal{U}(\mathcal{PO}(\mathcal{V})) \subseteq \mathcal{O}(\mathcal{U})$. The reverse inclusion is given by Lemma \ref{lem5.3}, which implies these two sets are equal. For any $\mathcal{W} \in \mathcal{P}art(X)$ finer than $\mathcal{V}$, we have 
    \[
        \mathcal{O}(\mathcal{U}) \subseteq \pi^\mathcal{W}_\mathcal{U}(\mathcal{PO}(\mathcal{W})) \subseteq \pi^\mathcal{V}_\mathcal{U}(\mathcal{PO}(\mathcal{V})) = \mathcal{O}(\mathcal{U}).
    \]
    Therefore, equality holds.

    Conversely, suppose that $f$ has the property stated above. Let $\mathcal{U} \in \mathcal{FOC}(X)$. Since $X$ is totally disconnected, let $\mathcal{U}' \in \mathcal{P}art(X)$ be a partition finer than $\mathcal{U}$. Let $\mathcal{V} \in \mathcal{P}art(X)$ be the partition witnessing the condition with respect to $\mathcal{U}'$. Given a $\mathcal{V}$-pseudo-orbit $\{x_i\}$ with pseudo-orbit pattern $(V_i)$, we have $\pi^\mathcal{V}_{\mathcal{U}'}((V_i)) \in \pi^\mathcal{V}_{\mathcal{U}'}(\mathcal{PO}(\mathcal{V}))$. By the hypothesis, $\pi^\mathcal{V}_{\mathcal{U}'}(\mathcal{PO}(\mathcal{V})) = \mathcal{O}(\mathcal{U}')$, so there exists $(U'_i) \in \mathcal{O}(\mathcal{U}')$ such that $\pi^\mathcal{V}_{\mathcal{U}'}((V_i)) = (U'_i)$, which implies $V_i \subseteq U'_i$. 
    Let $z \in [U'_i]$ (the set of points realizing the orbit pattern). Then there exists a $z$-orbit $(z_i)$ with orbit pattern $(U_i')$. We have
    \[
        z_i, x_i \in U_i' \subseteq U_i \quad (\text{for some } U_i \in \mathcal{U}),
    \]
    which means that each $\mathcal{V}$-pseudo-orbit is $\mathcal{U}$-shadowed by some $f$-orbit. Therefore, $f$ has the multivalued shadowing property.
\end{proof}
The following procedure is nearly the same as that in \cite{GM}, but for completeness, we write it out.
\begin{thm}\label{thm5.6}
    Let $F \colon X \to 2^X$ be a multivalued map with the multivalued shadowing property on a compact totally disconnected Hausdorff space $X$. Let $\{\mathcal{U}_\lambda\}$ be a cofinal directed subset of $\mathcal{P}art(X)$. Then 
    \begin{enumerate}
        \item the system 
            $\left( \varprojlim \{\mathcal{O}(\mathcal{U}_\lambda), \pi^\eta_\lambda\}, \varprojlim \sigma_\lambda \right)$
        is conjugate to 
           $ \left( \varprojlim \{\mathcal{PO}(\mathcal{U}_\lambda), \pi^\eta_\lambda\}, \varprojlim \sigma_\lambda \right);$
        
        \item both inverse systems satisfy the Mittag-Leffler condition.
    \end{enumerate}
\end{thm}
\begin{proof}
    First, given $\lambda$, $\mathcal{O}(\mathcal{U}_\lambda)$ is a subset of $\mathcal{PO}(\mathcal{U}_\lambda)$. Hence, the map 
    \[
        j_* \colon \varprojlim \mathcal{O}(\mathcal{U}_\lambda) \to \varprojlim \mathcal{PO}(\mathcal{U}_\lambda)
    \]
    induced by the inclusion maps is a continuous injection and commutes with $\varprojlim \sigma_\lambda$.

    Next, we show that $j_*$ is a surjective map. We define a monotone map $p \colon \Lambda \to \Lambda$ such that for each $\lambda \in \Lambda$, we have $p(\lambda) \geq \lambda$ and $\pi^{\mathcal{U}_{p(\lambda)}}_{\mathcal{U}_\lambda}(\mathcal{PO}(\mathcal{U}_{p(\lambda)})) = \mathcal{O}(\mathcal{U}_\lambda)$. We define a map 
    \[
        \phi \colon \varprojlim \{\mathcal{PO}(\mathcal{U}_\lambda), \pi^\eta_\lambda\} \to \varprojlim \{\mathcal{O}(\mathcal{U}_\lambda), \pi^\eta_\lambda\}
    \]
    as follows:
    \[
        (\phi(w))_\lambda = \pi^{\mathcal{U}_{p(\lambda)}}_{\mathcal{U}_\lambda}(w_{p(\lambda)}), \quad \text{for } w = (w_\gamma).
    \]
    This map is well-defined, continuous, and commutes with $\varprojlim \sigma_\lambda$.

    Now, consider the composition $j_* \circ \phi$. For a point $w = (w_\gamma) \in \varprojlim \{\mathcal{PO}(\mathcal{U}_\lambda), \iota\}$, we have
    \[
        (j_*(\phi(w)))_\lambda = j_\lambda(\pi^{\mathcal{U}_{p(\lambda)}}_{\mathcal{U}_\lambda}(w_{p(\lambda)})) = \pi^{\mathcal{U}_{p(\lambda)}}_{\mathcal{U}_\lambda}(w_{p(\lambda)}) = w_\lambda,
    \]
    where the last equality holds because $w$ is a thread in the inverse limit and $p(\lambda) \geq \lambda$.
    In particular, $j_* \circ \phi$ is the identity on $\varprojlim \{\mathcal{PO}(\mathcal{U}_\lambda), \pi^\eta_\lambda\}$. Since $j_*$ is injective, it follows that both $j_*$ and $\phi$ are conjugacies.

    It remains to show that both systems satisfy the Mittag-Leffler condition. By Lemma \ref{lem5.3}, the inclusion maps are surjective; thus, the system $\{\mathcal{O}(\mathcal{U}_\lambda), \pi^\eta_\lambda\}$ satisfies the Mittag-Leffler condition. For the system $\{\mathcal{PO}(\mathcal{U}_\lambda), \pi^\eta_\lambda\}$, we proceed as follows. Let $\lambda \in \Lambda$ and choose $\gamma \geq \lambda$ so that $\mathcal{U}_\gamma$ witnesses the shadowing property for $\mathcal{U}_\lambda$. Then for any $\eta \geq \gamma$, by Lemma \ref{lem5.5}, we have 
    \[
        \pi^\eta_\lambda(\mathcal{PO}(\mathcal{U}_\eta)) = \mathcal{O}(\mathcal{U}_\lambda) = \pi^\gamma_\lambda(\mathcal{PO}(\mathcal{U}_\gamma)).
    \]
    Thus, the system $\{\mathcal{PO}(\mathcal{U}_\lambda)\}$ satisfies the Mittag-Leffler condition.
\end{proof}

\begin{cor}\label{cor5.7}
    Let $f \colon X \to 2^X$ be a multivalued map with the multivalued shadowing property on a compact totally disconnected Hausdorff space $X$. Then $(X,f)$ is conjugate to an inverse limit of multivalued maps of finite type such that the inverse systems of the induced orbit systems satisfy the Mittag-Leffler condition.
\end{cor}

\begin{proof}
    This is a direct consequence of Lemma \ref{lem4.7}, Theorem \ref{thm4.11}, and Theorem \ref{thm5.6}.
\end{proof}

\begin{thm}\label{thm6.8}
    Let $f:X\to 2^X$ be a multivalued map on a compact totally disconnected Hausdorff space $X$. Then $f$ has shadowing property if and only if $(X,f)$ is conjugate to an inverse limit of multivalued of finite type such that the inverse system of their induced orbit systems satisfy Mittag-Leffler condition. 
\end{thm}
\begin{proof}
    This is a consequence of Theorem \ref{thm3.12} and Corollary \ref{cor5.7}.
\end{proof}
For single-valued maps, recall the profound result proved in \cite{GM}.

\begin{thm}\label{thm6.9}
    Let $f:X\to X$ be a map on a compact totally disconnected Hausdorff space $X$. Then $f$ has single-valued shadowing if and only if $(X,f)$ is conjugate to the inverse limit of an ML inverse system of subshifts of finite type.
\end{thm}

Combining the results above, we obtain the following statement.

\begin{thm}
    Let $f:X\to 2^X$ be a multivalued map with shadowing property on a compact totally disconnected Hausdorff space $X$. Then its orbit system $({\rm{Orb}}_f(X),\sigma_f)$ has single-valued shadowing property.
\end{thm}
\begin{proof}
    This is a consequence of Theorem \ref{thm5.6}, Theorem \ref{thm6.9} and Theorem \ref{thm6.8}.
\end{proof}

If $X$ is a compact totally disconnected Hausdorff space, we have the following observation.

\begin{cor}
    Let $f:X\to 2^X$ be a multivalued map on the Cantor set or any compact totally disconnected metric space $X$. Then $f$ has multivalued shadowing property if and only if $(X,f)$ is conjugate to the inverse limit of an ML sequence of multivalued maps of finite type such that their orbit systems satisfy Mittag-Leffler condition.
\end{cor}

In~\cite{BBK}, the authors asked whether there exists a characterization of shadowing in terms of graph covers (see~\cite[Questions]{BBK}). As an application of single-valued maps on totally disconnected spaces, we establish the following result.

\begin{thm}
    Let $\mathcal{G} = \{G_1 \xleftarrow{\varphi_1} G_2 \xleftarrow{\varphi_2} \cdots\}$ be a graph cover. The system has the shadowing property if and only if the inverse system of the orbit spaces (infinite walk) of $G_i=(V_i, E_i)$ ($i\in \mathbb{N}$) satisfies the Mittag-Leffler condition.
\end{thm}

\begin{proof}
    This result follows directly from Theorem~\ref{thm3.12} and Theorem~\ref{thm6.8}.
\end{proof}

\section{Density of multivalued shadowing in totally disconnected spaces}\label{se7}

Now, we shift our focus from single-valued continuous functions to the category of multivalued semi-continuous functions.

\subsection{Class of multivalued maps}

Let $X$ and $Y$ be two compact metric spaces. We denote by $C(X,Y)$ the set of all single-valued continuous maps from $X$ to $Y$. Suppose that $d'$ is a metric on $Y$. We define the uniform metric $d_\infty'$ on $C(X,Y)$ by
\[
    d_\infty'(f,g) = \sup_{x\in X} d'(f(x),g(x)) \quad \text{for all } f,g\in C(X,Y).
\]

Let $(2^Y, d_H')$ be the hyperspace of $Y$. We denote by $UC(X,Y)$ the set consisting of all upper semi-continuous maps from $X$ to $2^Y$. 
Similarly, we can define $LC(X)$ as the set of lower semi-continuous multivalued maps on $X$. We know that 
\[
    LC(X) \cap UC(X) = C(X, 2^X).
\]

Let $d$ be a metric on $X$. Recall that $d_H$ denotes the Hausdorff metric on $2^X$. For two compact-valued maps $f, g \in UC(X)$, we define
\[
    d_H^\infty(f,g) = \sup_{x\in X} d_H(f(x), g(x)).
\]
\begin{lem}
Let $(X,d)$ be a compact metric space. Suppose that $\{f_i\}_{i=1}^{\infty}$ is a Cauchy sequence of upper semi-continuous (resp.\ lower semi-continuous) maps from $X$ to $2^X$ with respect to the metric $d^\infty_H$, which converges to a map $f \colon X \to 2^X$. Then $f$ is also upper semi-continuous (resp.\ lower semi-continuous). Consequently, the spaces $(UC(X),d_H^\infty)$ and $(LC(X),d_H^\infty)$ are complete.
\end{lem}
\begin{proof}
    For each $x \in X$, the sequence $\{f_i(x)\}_{i=1}^\infty$ converges to some $K_x \in 2^X$ in the Hausdorff metric. Thus, the limit map $f$ takes the form
    \[
        f(x) = K_x = \lim_{i \to \infty} f_i(x).
    \]

    We first show that $f$ is upper semi-continuous. Given $\varepsilon > 0$, there exists $N \in \omega$ such that for all $i > N$,
    \[
        \sup_{x \in X} d_H^\infty(f_i(x), f(x)) < \frac{\varepsilon}{3}.
    \]
    This implies
    \begin{align} \label{eq7.1}
        f_i(x) \subset B_d\left(f(x), \frac{\varepsilon}{3}\right) \quad \text{and} \quad f(x) \subset B_d\left(f_i(x), \frac{\varepsilon}{3}\right)
    \end{align}
    for all $x \in X$ and $i > N$.

    Fix $x \in X$ and $i > N$. Since $f_i$ is upper semi-continuous, there exists $\delta_x > 0$ such that
    \[
        f_i(B_d(x, \delta_x)) \subset B_d\left(f_i(x), \frac{\varepsilon}{3}\right) \subset B_d\left(f(x), \frac{2\varepsilon}{3}\right).
    \]
    Now, for each $y \in f(B_d(x, \delta_x))$, \eqref{eq7.1} implies that there exists $z \in f_i(B_d(x, \delta_x))$ with $d(y, z) < \varepsilon/3$,  which means that 
    \[y\in B_d\left(f(x),\varepsilon\right)\text{ and thus }f\left(B_d(x,\delta_x)\right)\subset B_d\left(f(x),\varepsilon\right).\]
    Since $\varepsilon > 0$ is arbitrary, $f$ is upper semi-continuous.

   If the sequence $\{f_i\}_{i=1}^\infty$ are lower semi-continuous, we now show that $f$ is lower semi-continuous. Given $\varepsilon > 0$, there exists $N \in \omega$ such that for all $i > N$,
    \[
        \sup_{x \in X} d_H^\infty(f_i(x), f(x)) < \frac{\varepsilon}{3}.
    \]
    Fix $x \in X$ and $i > N$. For each $y \in f(x)$, the above implies
    \[
        B_d\left(y, \frac{2\varepsilon}{3}\right) \cap f_i(x) \neq \emptyset.
    \]
    Since $f_i$ is lower semi-continuous, there exists $\delta_x > 0$ such that for all $z \in B_d(x, \delta_x)$,
    \[
        f_i(z) \cap B_d\left(y, \frac{2\varepsilon}{3}\right) \neq \emptyset.
    \]
    For each $z' \in f_i(z) \cap B_d\left(y, \frac{2\varepsilon}{3}\right)$, \eqref{eq7.1} implies that there exists $z'' \in f(z)$ with $d(z', z'') < \varepsilon/3$. By the triangle inequality,
    \[
        d(y, z'') \leq d(y, z') + d(z', z'') < \frac{2\varepsilon}{3} + \frac{\varepsilon}{3} = \varepsilon.
    \]
    Thus,
    \[
        f(z) \cap B_d(y, \varepsilon) \neq \emptyset \quad \text{for all } z \in B_d(x, \delta_x).
    \]
    Since $\varepsilon > 0$ is arbitrary, $f$ is lower semi-continuous.
\end{proof}

\subsection{Multivalued shadowing is dense in \texorpdfstring{$UC(X)$}{UC(X)}}
We prove that the collection of multivalued maps with the shadowing property is dense in the space of all multivalued maps on a compact totally disconnected Hausdorff space. It is also considered in like \cite{Sh89}, \cite{BD}

Throughout this subsection, let $X$ be a compact totally disconnected Hausdorff space. First, we recall the definition of the \textit{Vietoris topology} on the hyperspace $2^X$. Given open subsets $U_1, \dots, U_n$ of $X$, we define:
\[
    \langle U_1, \dots, U_n \rangle = \left\{ A \in 2^X : A \subset \bigcup_{i=1}^n U_i \quad \text{and} \quad A \cap U_i \neq \emptyset \text{ for all } 1 \leq i \leq n \right\}.
\]
The collection of all sets of the form $\langle U_1, \dots, U_n \rangle$ constitutes a basis for the Vietoris topology on $2^X$ (see \cite{AF,IN}).

Inspired by the basis of the Vietoris topology, we introduce the following definition. Let $\mathcal{U} \in \mathcal{P}art(X)$ be a finite partition of $X$. For $f, g \in UC(X)$, we say that $g$ is \textit{$\mathcal{U}$-close} to $f$ if for each $x \in X$, the following holds:
\[
    f(x) \subset \bigcup_{i \in I} U_i \quad \text{and} \quad f(x) \cap U_i \neq \emptyset \ (\forall i \in I) \implies g(x) \subset \bigcup_{i \in I} U_i \quad \text{and} \quad g(x) \cap U_i \neq \emptyset \ (\forall i \in I),
\]
where $\{U_i\}_{i \in I}$ is the sub-collection of $\mathcal{U}$ intersecting $f(x)$.

Thus, we have the following statement.

\begin{lem}
    Let $\{f_\lambda\}_{\lambda \in \Lambda}$ be a net of multivalued maps $f_\lambda: X \to 2^X$. Suppose that for any $\mathcal{U} \in \mathcal{P}art(X)$, there exists $\eta \in \Lambda$ such that for all $\gamma \geq \eta$, $f_\gamma$ is $\mathcal{U}$-close to $F_\eta$. Then the net $f_\lambda$ converges pointwise to a multivalued map $f$, and $f$ is upper semi-continuous.
\end{lem}

\begin{proof}
    First, we show the pointwise convergence. For each fixed $x \in X$, the given condition implies that $\{f_\lambda(x)\}_{\lambda \in \Lambda}$ forms a Cauchy net with respect to the Vietoris topology on $2^X$. Since $2^X$ is compact, this net converges to some compact set, which we denote by $f(x)$.

    Next, we prove that $f$ is upper semi-continuous. Let $x \in X$ and let $W$ be an open neighborhood of $f(x)$. Since $f_\lambda(x) \to f(x)$, there exists $\eta \in \Lambda$ such that $f_\gamma(x) \subset W$ for all $\gamma \geq \eta$. Fix such a $\gamma$. Since $f_\gamma$ is upper semi-continuous, there exists an open neighborhood $V$ of $x$ such that $f_\gamma(y) \subset W$ for all $y \in V$.
    
    By the construction of the limit (and the fineness of partitions), we can infer that $f(y) \subset W$ for all $y \in V$. This implies that $f$ is upper semi-continuous.
\end{proof}

\begin{lem}\label{lem7.4}
Given $f:X\to 2^X$ a multivalued map on $X$. For any $\mathcal{U}\in \mathcal{P}art(X)$ there is a multivalued map with shadowing property $f^\mathcal{U}:X\to 2^X$ such that $f^\mathcal{U}$ is $\mathcal{U}$-close to $f$.
\end{lem}
\begin{proof}
    For each $x\in X$, we define
    \[f^\mathcal{U}(x)=\bigcup\{U\in \mathcal{U}:U\cap F(x)\ne\emptyset\}.\]
    As $f$ is upper semicontinuous, it is easy to verify that $f^\mathcal{U}$ is also upper semicontinuous and of course with compact value.

    Given $\mathcal{W}\in \mathcal{P}art(X)$, let $\mathcal{V}\in \mathcal{P}art(X)$ with $\mesh(\mathcal{V})<dist(\mathcal{U})$, where 
    \[dist(\mathcal{U})=\min\{d(U_i,U_j):U_i\cap U_j=\emptyset\}\]
    and $\{x_i\}$ a $\mathcal{V}$-pseudo orbit, that is, $x_i',x_{i+1}\in V_{i+1}$ for some $x_i'\in f^\mathcal{U}(x_i)$, from the construction of $f^\mathcal{U}$, $x_{i+1}\in f^\mathcal{U}(x_i)$ for each $i\in \omega$, which means that $(x_i)$ is actually an $f^\mathcal{U}$-orbit. Therefore, $f^\mathcal{U}$ has shadowing property.
\end{proof}

As a consequence ,we obtain:

\begin{thm}
    Let $f:X\to 2^X$ be a multivalued map on compact totally disconnected metric space $(X,d)$ and $\varepsilon>0$. We denote by $d_H$ the \textit{Hausdorff metric} induced by $d$. Then there is a multivalued map $g$ with shadowing property such that 
    \[d_H(f(x),g(x))<\varepsilon\text{ for all }x\in X.\]
    Therefore, multivalued maps with shadowing property is dense in the collection of all multivalued maps.
\end{thm}

\section{Shadowing for general multivalued systems}\label{se8}
In this section, inspired by \cite{GM}, we discuss the multivalued shadowing property for general metric spaces. 
First, let $\mathcal{FOC}(X)$ denote the collection of finite open covers of $X$. For any cover $\mathcal{C}$ of $X$ and a subset $A \subset X$, the \textit{star} of $A$ with respect to $\mathcal{C}$ is defined as 
\[
st(A, \mathcal{C}) = \bigcup \{ U \in \mathcal{C} : U \cap A \neq \emptyset \}.
\]
For each $\mathcal{U} \in \mathcal{FOC}(X)$, we denote its \textit{mesh} by $\operatorname{mesh}(\mathcal{U}) = \sup \{ \operatorname{diam}(U) : U \in \mathcal{U} \}$ and its \textit{Lebesgue number} by $\operatorname{Leb}(\mathcal{U})$.

\begin{thm}\label{thm7.1}
Let $X$ be a compact metric space and $f \colon X \to 2^X$ a multivalued map with the shadowing property. Then there exists a sequence $(g_n^{n+1}, X_n)$ of multivalued maps of finite type such that $(X, f)$ is a subfactor of 
$\left( \varprojlim (g_n^{n+1}, X_n), \varprojlim f_n \right),$
where each $f_n \colon X_n \to 2^{X_n}$ is a multivalued map of finite type for $n \in \omega$. 
Moreover, its orbit system $(\operatorname{Orb}_f(X), \sigma_f)$ is a factor of the inverse limit of the orbit systems of these multivalued maps; that is, a factor of an inverse limit of subshifts of finite type.
\end{thm}
\begin{proof}
        Let $\{\mathcal{U}_i\}$ be a sequence of finite open covers satisfying the following properties:
    \begin{enumerate}
        \item $\mathcal{U}_{n+1}$ witness $\mathcal{U}_n$ shadowing;
        \item $\{\mathcal{U}_i\}$ is cofinal in $\mathcal{FOC}(X)$, and
        \item for all $U\in \mathcal{U}_{n+2}$, there is $W\in \mathcal{U}_n$ such that $st(U,\mathcal{U}_{n+1})\subset W$.
        \item For each $U\in\mathcal{U}_n$, it can not be contained in the union of other open subsets of $\mathcal{U}_n$.
    \end{enumerate}

    First, let $\mathcal{U}_0 = \{X\}$. We proceed by induction: suppose $\mathcal{U}_n$ is defined, and let $\mathcal{U}_{n+1}$ be a finite open cover witnessing the multivalued shadowing property for $\mathcal{U}_n$, such that 
\[
\operatorname{mesh}(\mathcal{U}_{n+1}) < \frac{1}{3} \operatorname{Leb}(\mathcal{U}_n).
\]
This construction ensures that conditions 1. and 2. are satisfied. 

For 3., fix $n \in \mathbb{N}$ and $U \in \mathcal{U}_{n+2}$. By the construction, $U \subset V$ for some $V \in \mathcal{U}_{n+1}$, which implies $st(U, \mathcal{U}_{n+1}) \subset st(V, \mathcal{U}_{n+1})$. Note that 
\[
\operatorname{diam}(st(V, \mathcal{U}_{n+1})) \leq 3 \operatorname{mesh}(\mathcal{U}_{n+1}) < \operatorname{Leb}(\mathcal{U}_n).
\]
Consequently, by the definition of the Lebesgue number, there exists some $W \in \mathcal{U}_n$ such that 
\[
st(U, \mathcal{U}_{n+1}) \subset st(V, \mathcal{U}_{n+1}) \subset W,
\]
as required. Condition 4. can be easily done.

Let $f \colon X \to 2^X$ and the sequence of covers $\{\mathcal{U}_i\}$ be as defined above. For each $U \in \mathcal{U}_{n+2}$, fix an element $W(U) \in \mathcal{U}_n$ such that $st(U, \mathcal{U}_{n+1}) \subset W(U)$. For any $\mathcal{U} \in \mathcal{FOC}(X)$, we define the induced map $f_{\mathcal{U}} \colon \mathcal{U} \to 2^{\mathcal{U}}$ by 
\[
V \in f_{\mathcal{U}}(W) \iff f(W) \cap V \neq \emptyset.
\]
We define the bonding map $w \colon \mathcal{U}_{n+2} \to \mathcal{U}_n$ by $U \mapsto W(U)$. It is straightforward to verify that 
\[
V \in f_{\mathcal{U}_{n+2}}(U) \implies W(V) \in f_{\mathcal{U}_n}(W(U)).
\]
Thus, $w$ satisfies Condition \eqref{eq1.11}. Furthermore, by property 4., the sequence of multivalued systems $\{(w, (\mathcal{U}_{2n}, f_{\mathcal{U}_{2n}}))\}$ satisfies the Mittag-Leffler condition. By Theorem \ref{lem1.12}, we obtain the inverse limit system $(\varprojlim (\mathcal{U}_{2n}, w), \varprojlim f_{\mathcal{U}_{2n}})$ and define the map
\[
g \colon (U_n)_{n \in \omega} \to \bigcap_{n \in \omega} U_n.
\]
Since $\{\mathcal{U}_n\}$ is a cofinal sequence of covers, $g$ is well-defined. Suppose $(V_n) \in \varprojlim f_{\mathcal{U}_{2n}}((U_n))$ with $\bigcap U_n = \{x\}$ and $\bigcap V_n = \{y\}$. To show $g$ is a morphism, assume for contradiction that $y \notin f(x)$. Then there exist neighborhoods $O_1, O_2$ of $x, y$ respectively such that $f(O_1) \cap O_2 = \emptyset$, which contradicts the definition of $f_{\mathcal{U}_{2n}}$. Therefore, $g$ satisfies Condition \eqref{eq1.11}, and $(X,f)$ is a subfactor of $\varprojlim (\mathcal{U}_{2n}, w)$.

Moreover, we claim that the orbit system $(\operatorname{Orb}_f(X), \sigma_f)$ is a factor of the inverse limit of the orbit systems of these multivalued maps of finite type. For convenience, we adopt the notation established in Section \ref{subse6.1}.

\begin{cla}
$w(\mathcal{PO}(\mathcal{U}_{n+2})) \subset \mathcal{O}(\mathcal{U}_n)$.
\end{cla}

\begin{proof}
Let $(U_j) \in \mathcal{PO}(\mathcal{U}_{n+2})$ and let $(x_j)$ be a pseudo-orbit with pattern $(U_j)$. Since $\mathcal{U}_{n+2}$ witnesses the $\mathcal{U}_{n+1}$-shadowing property, there exists an $f$-orbit $(z_j)$ and a sequence $(V_j) \in \mathcal{O}(\mathcal{U}_{n+1})$ such that $z_j, x_j \in V_j$ for all $j$. 

In particular, by the choice of $(V_j)$ and the definition of the star operator, we have $V_j \subset st(U_j, \mathcal{U}_{n+1}) \subset W(U_j)$. This implies that each $W(U_j)$ is non-empty, which establishes that $(W(U_j)) \in \mathcal{O}(\mathcal{U}_n)$. Although $w$ is not necessarily surjective, for every $(x_j) \in \operatorname{Orb}_f(X)$, the orbit $(x_j)$ is itself a $\mathcal{U}_{n+2}$-pseudo-orbit. Thus, there exists some $(U_j) \in \mathcal{PO}(\mathcal{U}_{n+2})$ such that $x_j \in U_j$, and consequently $x_j \in W(U_j)$, confirming that $(W(U_j))$ is a $\mathcal{U}_n$-orbit pattern for $(x_j)$.
\end{proof}

\begin{cla}
The orbit space $(\operatorname{Orb}_f(X), \sigma_f)$ is a factor of the inverse limit of the orbit systems of multivalued maps of finite type.
\end{cla}

\begin{proof}
First, define $\iota' \colon \mathcal{PO}(\mathcal{U}_{n+2}) \to \mathcal{PO}(\mathcal{U}_n)$ by $\iota' = j \circ w$, where $j \colon \mathcal{O}(\mathcal{U}_n) \hookrightarrow \mathcal{PO}(\mathcal{U}_n)$ is the inclusion map. Similarly, define $\iota'' \colon \mathcal{O}(\mathcal{U}_{n+2}) \to \mathcal{O}(\mathcal{U}_n)$ by $\iota'' = w \circ j'$, where $j' \colon \mathcal{O}(\mathcal{U}_{n+2}) \hookrightarrow \mathcal{PO}(\mathcal{U}_{n+2})$ is the inclusion. Thus, $(\mathcal{PO}(\mathcal{U}_{2i}), \iota')$ and $(\mathcal{O}(\mathcal{U}_{2i}), \iota'')$ form two inverse systems. The map $w$ induces a morphism $w^*$ from $\varprojlim (\mathcal{PO}(\mathcal{U}_{2i}), \iota')$ to $\varprojlim (\mathcal{O}(\mathcal{U}_{2i}), \iota'')$ defined by $w^*(\{U_i\}) = \{w(U_i)\}$, which commutes with the shift map $\sigma^*$.

By Theorem \ref{thm5.2}, we define a map 
\[
\phi \colon \varprojlim (\mathcal{O}(\mathcal{U}_{2i}), \iota'') \to \operatorname{Orb}_f(X) \quad \text{by} \quad \phi(u_i) = \left( \bigcap_{i} \overline{\pi_j(u_i)} \right).
\]
In view of Theorem \ref{thm5.2} and arguments similar to Theorem \ref{thm6.4}, $\phi$ is continuous and commutes with $\sigma^*$. To show that $\phi$ is surjective, fix $(x_j) \in \operatorname{Orb}_f(X)$. For each $i \in \mathbb{N}$, let $O_{2i}((x_j)) \subset \mathcal{O}(\mathcal{U}_{2i})$ be the set of $\mathcal{U}_{2i}$-orbit patterns for $(x_j)$, and let 
\[
O(\sigma_f^k(x_j)) = \bigcap_{i} \pi^{-1}_{2i}(\overline{\sigma^k(O((x_j)))}) \cap \varprojlim (\mathcal{O}(\mathcal{U}_{2i}), \iota'').
\]
Then $\sigma^k_f(O((x_j)))$ is non-empty for all $k \in \mathbb{N}$ and $\phi(\sigma^k_f(O((x_j)))) = (x_{j+k})$, which implies that $\phi$ is surjective. Therefore, $(\operatorname{Orb}_f(X), \sigma_f)$ is a factor of the inverse system via the factor map $\phi \circ w^*$.
\end{proof}
\end{proof}

A map $\pi:X\to Y$ is \textbf{open} if $f(U)$ is open whenever $U$ is open in $X$. Then we end this section by giving an observation.
\begin{thm}\label{thm8.4}
    Let $f:X\to 2^X$ and $g:Y\to 2^Y$ be two upper semicontinuous maps and $\pi:X\to Y$ an open factor between $(X,f)$ and $(Y,g)$. If $f$ has multivalued shadowing property, then $g$ has multivaluued shadowing property.
\end{thm}
\begin{proof}
    Let $\mathcal{U}_Y\in \mathcal{FOC}(Y)$. Then $\mathcal{U}_X=\pi^{-1}(\mathcal{U}_Y)$ is an open cover of $X$, let $\mathcal{V}_X\in \mathcal{FOC}(X)$  witness the shadowing property of $\mathcal{U}_X$ and take $\mathcal{V}_Y=\pi(\mathcal{V}_X)$. Suppose $\{y_i\}$ is a sequence of $\mathcal{V}_Y$-pseudo-orbit in $Y$, then there is a sequence $\{x_i\}$ such that it is a $\mathcal{V}_X$-pseudo-orbit and $\pi(x_i)=y_i$. As $(X,f)$ has shadowing property, there is an $\mathcal{U}_X$-orbit $(z_i)$ such that $(x_i,z_i)\in U_i$ for some $U_i\in \mathcal{U}_X$. Thus, $(y_i(=\pi(x_i)),\pi(z_i))\in \pi(U_i)\in \mathcal{U}_Y$, which means that each $\mathcal{V}_Y$-pseudo-orbit is shadowed by some orbit in $Y$ with $\mathcal{U}_Y$ pattern. Therefore, $g$ has shadowing property.
\end{proof}

Let $f:X\to 2^X$ be an upper semicontinuous map. For each $n\in \omega$, we define
\[{\rm{Orb}}^n_f(X):=\left\{(x_1,\cdots,x_n)\in \prod_{i=1}^nX:x_{i+1}\in f(x_i), i=1,\cdots,n-1\right\}\]
the $n$-length orbit space.
Due to Theorem \ref{thm8.4}, we can directly prove the following statement.
\begin{cor}
    Let $f:X\to 2^X$ be an upper semicontinuous map. If $f$ is lower semicontinuous, then 
    \begin{enumerate}
        \item the first projection map $\Pi_0:{\rm{Orb}}_f(X)\to X$ is open;
        \item $f$ has multivalued shadowing property if and only if $({\rm{Orb}}_f(X),\sigma_f)$ has shadowing property.
    \end{enumerate}
\end{cor}
\begin{proof}
  1.  In the absence of ambiguity, for each $n\in \mathbb{N}$ and $\mathcal{U}\in \mathcal{FOC}(X)$, the set 
    \[[U_1,\cdots,U_n]=\{(x_1,\cdots,x_n):x_i\in U_i, x_{i+1}\in f(x_i) \}\]
    is an open subset in ${\rm{Orb}}_f^n(X)$, and one can verify that 
    \[\Pi_0([U_1,\cdots,U_n])=U_1 \cap f^{-1}\left( U_2 \cap f^{-1}\left( U_3 \cap \cdots \cap f^{-1}(U_n) \right) \right).\]
    As $f$ is lower semicontinuous, $f^{-1}(U)$ is open whenever $U$ is open in $X$, $\Pi_0([U_1,\cdots,U_n])$ is an open subset in $X$. Since the class $\left\{[U_1,\cdots,U_n]:n\in \omega, U_i\in \text{ some }\mathcal{U}\in  \mathcal{FOC}(X)\right\}$ forms a basis in ${\rm{Orb}}^n_f(X)$, $\Pi_0$ is open.

    2. On the one hand, if $({\rm{Orb}}_f(X),\sigma_f)$ has shadowing property, from Theorem \ref{thm8.4} and 1., we conclude that $(X,f)$ has mutivalued shadowing property.

    On the other hand, suppose $f$ has shadowing property, the proof is presented in \cite{Yin}.
\end{proof}
\section*{Acknowledgement}
The author is deeply grateful to Professor X. Dai and Professor D. Dou for their help and the insightful suggestions they provided during my time at Nanjing University.

\end{document}